\documentclass[english]{amsart}
\pdfoutput=1

\DeclareRobustCommand{\SkipTocEntry}[5]{}

\usepackage{graphicx}
\usepackage{amsmath,amsfonts,amsthm}
\usepackage{amssymb}
\usepackage{multirow,array}
\usepackage{comment}
\usepackage{enumitem}
\usepackage[T1]{fontenc}
\usepackage{emptypage}
\usepackage{hyperref}
\usepackage{color}

\usepackage[url=true,
doi=false,
maxbibnames=99,
style=numeric,
backend=biber,
natbib=true,
sorting=nty]{biblatex}
\renewbibmacro{in:}{}
\bibliography{AllRefs}

\numberwithin{equation}{section}
\newcounter{dummy}
\numberwithin{dummy}{section}

\newtheorem{theorem}[dummy]{Theorem}

\newtheorem{corollary}[dummy]{Corollary}

\newtheorem{definition}[dummy]{Definition}

\newtheorem{lemma}[dummy]{Lemma}

\newtheorem{proposition}[dummy]{Proposition}

\theoremstyle{remark}
\newtheorem{remark}[dummy]{Remark}

\theoremstyle{definition}
\newtheorem{example}[dummy]{Example}

\newcommand{\supp}{\textnormal{supp}}
\newcommand{\Lip}{\textnormal{Lip}}

\newcommand{\DD}{\mathcal{D}}

\newcommand{\FF}{\mathcal{F}}
\newcommand{\GG}{\mathcal{G}}

\newcommand{\UU}{\mathcal{U}}
\newcommand{\MM}{\mathcal{M}}

\newcommand{\R}{\mathbb R}

\newcommand{\N}{\mathbb N}

\newcommand{\X}{\mathbf X}
\newcommand{\Y}{\mathbf Y}
\newcommand{\x}{\mathbf x}
\newcommand{\y}{\mathbf y}
\newcommand{\z}{\mathbf z}

\newcommand{\B}{\mathbf B}
\newcommand{\ZZ}{\mathbf Z}

\newcommand{\var}[1]{#1\textnormal{-var}}

\newcommand{\pvar}{p\textnormal{-var}}
\newcommand{\pprimevar}{p'\textnormal{-var}}
\newcommand{\qvar}{q\textnormal{-var}}
\newcommand{\onevar}{1\textnormal{-var}}

\newcommand{\pHol}{1/p\textnormal{-H{\"o}l}}

\newcommand{\EEE}[1]{\mathbb E \left[ #1 \right]}
\newcommand{\Law}[1]{\textnormal{Law}\left[ #1 \right]}
\newcommand{\EEEover}[2]{\mathbb E^{#1} \left[ #2 \right]}
\newcommand{\PPPover}[2]{\mathbb P^{#1} \left[ #2 \right]}

\newcommand{\Diff}{\textnormal{Diff}}

\newcommand{\PPP}[1]{\mathbb{P} \left[ #1 \right]}

\newcommand{\id}{\textnormal{id}}
\newcommand{\norm}[1]{\left|\left|#1 \right|\right|}
\newcommand{\normt}[1]{\left|\left|\left|#1 \right|\right|\right|}
\newcommand{\norms}[1]{\left|#1 \right|}
\newcommand{\normc}{\norm{\cdot}}

\newcommand{\1}[1]{\mathbf 1 \{#1\}}

\newcommand{\LLL}{\mathbf L}

\newcommand{\Lyons}{S}

\newcommand{\uu}{\mathfrak u}
\newcommand{\g}{\mathfrak g}

\newcommand{\h}{\mathfrak h}

\newcommand{\floor}[1]{\lfloor #1 \rfloor}
\newcommand{\roof}[1]{\lceil #1 \rceil}

\newcommand{\midset}{\,\middle|\,}

\newcommand\restr[2]{{\left.\kern-\nulldelimiterspace 
  #1 
  \vphantom{\big|} 
  \right|_{#2} 
  }}

\def\eqd{\,{\buildrel \DD \over =}\,}

\def\convd{\,{\buildrel \DD \over \rightarrow}\,}

\makeatother

\begin{document}

\title{Random walks and L{\'e}vy processes as rough paths}

\author{Ilya Chevyrev}
\address{I. Chevyrev,
Mathematical Institute,
University of Oxford,
Andrew Wiles Building,
Radcliffe Observatory Quarter,
Woodstock Road,
Oxford OX2 6GG,
United Kingdom}
\email{chevyrev@maths.ox.ac.uk}
\thanks{The author is supported by a Junior Research Fellowship of St John's College, Oxford.}

\subjclass[2010]{Primary 60G51; Secondary 60H10}



\keywords{Homogeneous groups, rough paths, L{\'e}vy processes, random walks, tightness of $p$-variation, stochastic flows, characteristic functions of signatures}

\begin{abstract}
We consider random walks and L{\'e}vy processes in a homogeneous group $G$. For all $p > 0$, we completely characterise (almost) all $G$-valued L{\'e}vy processes whose sample paths have finite $p$-variation, and give sufficient conditions under which a sequence of $G$-valued random walks converges in law to a L{\'e}vy process in $p$-variation topology. In the case that $G$ is the free nilpotent Lie group over $\R^d$, so that processes of finite $p$-variation are identified with rough paths, we demonstrate applications of our results to weak convergence of stochastic flows and provide a L{\'e}vy--Khintchine formula for the characteristic function of the signature of a L{\'e}vy process. At the heart of our analysis is a criterion for tightness of $p$-variation for a collection of c{\`a}dl{\`a}g strong Markov processes.
\end{abstract}

\maketitle


\section{Introduction}

This paper focuses on several questions regarding L{\'e}vy processes and random walks in homogeneous groups, with a particular focus on applications to rough paths theory.
Let $G$ be a homogeneous group (in the sense of~\cite{FollandStein82}) equipped with a sub-additive homogeneous norm and the corresponding left-invariant metric. We can summarise the three main results of the paper as follows.

\begin{itemize}
\item (Theorem~\ref{thm_LevyFinitepVar}) Given a L{\'e}vy process $\X$ in $G$, we determine (almost) all values of $p > 0$ for which the sample paths of $\X$ have almost surely finite $p$-variation.

\item (Theorem~\ref{thm:convToLevy}) We give sufficient conditions for a sequence of (interpolated and reparametrised) random walks in $G$ to converge weakly to a (interpolated and reparametrised) L{\'e}vy process in $G$ in $p$-variation topology.

\item (Theorem~\ref{thm_LKFormula}) In the case that $G = G^N(\R^d)$, the step-$N$ free nilpotent Lie group over $\R^d$, we determine a L{\'e}vy--Khintchine formula for the characteristic function (in the sense of~\cite{ChevyrevLyons16}) of the signature of the random rough path constructed from a L{\'e}vy process in $G$.
\end{itemize}

We apply the second of these results in the context of rough paths to show weak convergence of stochastic flows in several examples. Notably, we provide a significant generalisation of a result of Kunita~\cite{Kunita95} and of a related result of Breuillard, Friz and Huesmann~\cite{Breuillard09}.

We take a moment to discuss how our work relates to the appearance of c{\`a}dl{\`a}g rough paths in the current literature. Friz and Shekhar~\cite{FrizShekhar12} recently introduced a broad extension of rough paths theory to the c{\`a}dl{\`a}g setting. Their work in particular generalises the notion of rough integration and RDEs and significantly extends earlier work of Williams~\cite{Williams01} who gave pathwise solutions to differential equations driven by L{\'e}vy processes in $\R^d$.

As a family of c{\`a}dl{\`a}g rough paths of particular interest, L{\'e}vy process in $G^N(\R^d)$ of finite $p$-variation for some $1 \leq p < N +1$, were studied in~\cite{FrizShekhar12}. Such L{\'e}vy $p$-rough paths bear a resemblance to Markovian rough paths constructed from subelliptic Dirichlet forms on $L^2(G^N(\R^d))$, first studied in~\cite{FrizVictoir08} and recently in~\cite{CassOgrodnik14, CO17, ChevyrevLyons16}, in the sense that both processes may be viewed as stochastic rough paths whose evolution depends entirely on its first $N$ iterated integrals.

The method we employ here to give meaning to c{\`a}dl{\`a}g rough paths is to connect left- and right-limits with continuous paths and treat the resulting object as a classical rough path. We therefore do not address directly the concept of a c{\`a}dl{\`a}g RDE in this paper, but emphasise that our methods relate closely to Marcus SDEs and that Theorems~\ref{thm_LevyFinitepVar} and~\ref{thm_LKFormula} can be seen as generalisations of two related results in~\cite{FrizShekhar12} (discussed further in Section~\ref{sec:LevyProcesses}). We mention however that the method of proof used for our main results, which is based on approximating a L{\'e}vy process by a sequence of random walks, is different to the methods used in~\cite{FrizShekhar12}.

We also point out that our methods treat general interpolations, which depend arbitrarily on the endpoints of jumps, on the same footing as the simpler linear interpolation used in Marcus SDEs. Examples of interest of such non-linear interpolations date back to the works of McShane~\cite{McShane72} and Sussman~\cite{Sussmann91} on approximations of Brownian motion (discussed further in Examples~\ref{ex_nonLinearInter} and~\ref{ex_pertubedWalk}), and recently in the work of Flint, Hambly and Lyons~\cite{Flint16}.

A crucial result for our analysis, which we believe to be of independent interest, is a criterion for tightness of $p$-variation of strong Markov processes taking values in a Polish space (Theorem~\ref{thm:pVarTightCrit}). This result is a generalisation of the main result of Manstavi{\v c}ius~\cite{Manstavicius04}, which provides a criterion for a strong Markov process to have sample paths of a.s. finite $p$-variation.
Our proof of Theorem~\ref{thm:pVarTightCrit} is a simplification of the stopping times technique adopted in~\cite{Manstavicius04}.

Finally, we mention that while most applications presented in this paper concern geometric rough paths, and thus only require consideration of the free nilpotent Lie group, we have attempted to make statements in their natural level of generality. In particular, we believe that our results may prove to be of interest for studying random walks and L{\'e}vy processes in the Butcher group, which correspond to branched rough paths in the sense of~\cite{Gubinelli10, HairerKelly15} (see also Remark~\ref{remark:branchedRPs} below).

\addtocontents{toc}{\SkipTocEntry}
\subsection*{Outline of the paper}
In Section~\ref{sec:LieGroupIidArrays} we discuss iid arrays and L{\'e}vy processes taking values in a general Lie group.
Our only contribution in this section is the construction of a sequence of random walks $(\X^n)_{n \geq 1}$ associated with a L{\'e}vy process $\X$ such that $\X^n \convd \X$ in the Skorokhod topology, and for which tightness of $p$-variation is simple to verify.
In Section~\ref{sec:Hom} we recall several preliminary facts about homogeneous groups and spaces of paths of finite $p$-variation.

Section~\ref{sec:tightRandWalks} is devoted to the proof of Theorem~\ref{thm_cadlagPathTight}, which shows tightness of $p$-variation for a collection of random walks in a homogeneous group. This is a central result used in the proofs of the three main aforementioned theorems, which we state and prove in Section~\ref{sec:LevyProcesses}. In Section~\ref{subsubsec:StochFlows} we also provide several applications of Theorem~\ref{thm:convToLevy} to weak convergence of stochastic flows.

In Appendix~\ref{appendix:PathFuncs} we introduce the concept of path functions, which serve to connect the left- and right-limits of c{\`a}dl{\`a}g paths, and collect several technical results used throughout Section~\ref{sec:LevyProcesses}.
In Appendix~\ref{appendix:InfpVar} we describe conditions under which sample paths of a L{\'e}vy process possess infinite $p$-variation (used to complete the proof of Theorem~\ref{thm_LevyFinitepVar}).

\addtocontents{toc}{\SkipTocEntry}
\subsection*{Acknowledgements}

Part of this work is contained in the author's D.Phil thesis written under the guidance of Prof. Terry Lyons, to whom the author is sincerely grateful. The author would also like to thank Atul Shekhar, Guy Flint, and Prof. Peter Friz for helpful conversations on this topic, and the anonymous referees whose suggestions helped to significantly improve this paper.

\section{Iid arrays and L{\'e}vy processes in Lie groups}
\label{sec:LieGroupIidArrays}

\subsection{Notation}

Throughout this section, we fix a Lie group $G$ with Lie algebra $\g$, and identify $\g$ with the space of left-invariant vector fields on $G$.
Let $u_1, \ldots, u_m$ be a basis for $\g$. We equip $\g$ with the inner product for which $u_1,\ldots, u_m$ is an orthonormal basis. For an element $y \in \g$ we write $y = \sum_{i=1}^m y^iu_i$. When $x$ is an element of a normed space, we denote its norm by $|x|$.

We further fix an open neighbourhood $U \subset G$ of the identity $1_G \in G$, such that $U$ has compact closure and $\exp: \g \mapsto G$ is a diffeomorphism from a neighbourhood of zero in $\g$ onto $U$. Let $\xi_i \in C^\infty_c(G, \R)$ be smooth functions of compact support such that $\log(x) = \sum_{i=1}^m \xi_i(x) u_i$ for all $x \in U$ (that is, $\xi_i$ provide exponential coordinates of the first kind on $U$). We denote $\xi : G \mapsto \g, \xi(x) = \sum_{i=1}^m \xi_i(x) u_i$.

For a metric space $E$, denote by $D([0,T], E)$ the space of c{\`a}dl{\`a}g functions $\x : [0,T] \mapsto E$ equipped with the Skorokhod topology (see, e.g.,~\cite[Section~12]{Billingsley99}). We shall use the symbol $o$ to denote spaces of paths whose starting point is the identity element $1_G$. For example $D_o([0,T],G)$ denotes the set of all $\x \in D([0,T],G)$ such that $\x_0 = 1_G$.

\subsection{Preliminaries on iid arrays and L{\'e}vy processes}\label{subsec:iidArrays}

An \emph{array} in $G$ is a sequence of a finite collection of $G$-valued random variables $\left(X_{n1},\ldots, X_{nn}\right)_{n \geq 1}$. We call the array \emph{iid} if, for every $n \geq 1$, $X_{n1},\ldots, X_{nn}$ are iid.
We will always suppose that an iid array $X_{nj}$ is infinitesimal, i.e., $\lim_{n \rightarrow \infty} \PPP{X_{n1} \notin V} = 0$ for every neighbourhood $V$ of $1_G$. Furthermore, for all $n \geq 1$ we let
\[
B_n := \EEE{\xi(X_{n1})} \in \g,
\]
and for all $i,j\in\{1,\ldots, m\}$
\[
A_n^{i,j} := \EEE{\xi_i(X_{n1})\xi_j(X_{n1})}.
\]

For a collection of elements $x_{1},\ldots, x_{n}$ in $G$, we define the \emph{associated walk} $\x \in D_o([0,1],G)$ by
\[
\x_{t} =
\begin{cases} 1_G &\mbox{if } t \in [0,n^{-1}) \\ 
x_{1}\ldots x_{\floor{tn}} &\mbox{if } t \in [n^{-1},1],
\end{cases}
\]
and for an array $X_{nj}$, we refer to the associated random walk $\X^n$ to mean the sequence of associated walks built from the collections $(X_{n1},\ldots,X_{nn})$.

Recall that a (left) L{\'e}vy process in $G$ is a $D_o([0,T], G)$-valued random variable $\X$ with independent and stationary (right) increments. We refer to Liao~\cite{Liao04} for further details.

We call a \emph{L{\'e}vy triplet} (or simply \emph{triplet}) a collection $(A, B, \Pi)$ of an $m\times m$ covariance matrix $(A^{i,j})_{i,j = 1}^m$, an element $B = \sum_{i=1}^m B^i u_i \in \g$, and a L{\'e}vy measure $\Pi$ on $G$ (see~\cite[p. 12]{Liao04}).

A classical theorem of Hunt~\cite{Hunt56} asserts that for every L{\'e}vy process $\X$ in $G$, there exists a unique triplet $(A,B,\Pi)$ such that the generator of $\X$ is given for all $f \in C^2_0(G)$ and $x \in G$ by
\begin{multline*}
\lim_{t \rightarrow 0} t^{-1}\EEE{f(x\X_t) - f(x)} = \sum_{i=1}^m B^i (u_i f)(x) + \frac{1}{2}\sum_{i,j = 1}^m A^{i,j} (u_i u_j f)(x) \\
+ \int_G \left[ f(xy) - f(x) - \sum_{i = 1}^m \xi_i(y)(u_i f)(x)\right] \Pi(dy).
\end{multline*}
Conversely, every L{\'e}vy triplet gives rise to a unique L{\'e}vy process.

We will heavily use a characterisation due to Feinsilver~\cite{Feinsilver78} of when a $G$-valued random walk converges in law to a Markov process as a $D_o([0,1], G)$-valued random variable.
The following is a special case of the main results of~\cite{Feinsilver78}.

\begin{theorem}[Feinsilver~\cite{Feinsilver78}]\label{thm_iidConvLevy}
Let $X_{nj}$ be an iid array of $G$-valued random variables and $\X^n$ the associated random walk. Denote by $F_n$ the probability measure on $G$ associated with $X_{n1}$. Let $\X$ be a L{\'e}vy process in $G$ with triplet $(A,B,\Pi)$. 

Then $\X^n \convd \X$ as $D_o([0,1],G)$-valued random variables if and only if
\begin{enumerate}[label={(\arabic*)}]
\item \label{point_FConv} $\lim_{n \rightarrow \infty} n F_n(f) = \Pi(f)$ for every $f \in C_b(G)$ which is identically zero on a neighbourhood of $1_G$,
\item \label{point_BConv} $\lim_{n \rightarrow \infty} n B_n = B$, and
\item \label{point_AConv} for all $i,j\in\{1,\ldots, m\}$,
\[
\lim_{n \rightarrow \infty} nA^{i,j}_n = A^{i,j} + \int_G \xi_i(x)\xi_j(x) \Pi(dx).
\]
\end{enumerate}
\end{theorem}

The following notion of a scaling function will be used throughout the paper.

\begin{definition}[Scaling function]\label{def_scalingFunc}
A continuous bounded function $\theta : G \rightarrow \R$ is called a \emph{scaling function} if
\begin{enumerate}[label={(\roman*)}]
\item $\theta(1_G) = 0$,
\item $\theta(x) > 0$ for all $x \neq 1_G$,
\item there exists $C > 0$ such that $|\xi|^2 \leq C\theta$, and
\item there exists $c > 0$ such that $\theta(x) > c$ for all $x \in G\setminus U$.
\end{enumerate}
Let $X_{nj}$ be an iid array in $G$. We say that $\theta$ \emph{scales} the array $X_{nj}$ if
\[
\sup_{n\geq 1} n\EEE{\theta(X_{n1})} < \infty.
\]
\end{definition}

The importance behind the above definition is that given a scaling function $\theta$ which scales $X_{nj}$, the rate with which $\theta$ decays at $1_G$ will determine the values of $p > 0$ for which the $p$-variation of the associated random walk is tight (Theorem~\ref{thm_cadlagPathTight}).

\begin{example}
In the case $G = \R^d$, the prototypical example of a scaling function is $1 \wedge |\cdot|^2$. For a general Lie group $G$, the example extends as follows: let $c > 0$ be sufficiently small such that $W := \{\exp(y) \mid y \in \g, |y| \leq c\}$ is contained in $U$.
Then
\begin{equation}\label{eq_largestScaleFunc}
\theta(x) := \1{x \in U}(c^2 \wedge |\xi(x)|^2) + c^2\1{x \notin U}
\end{equation}
is a scaling function.
\end{example}

\begin{remark}\label{remark:squaresScale}
Suppose that $\theta$ is defined by~\eqref{eq_largestScaleFunc} and that $X_{nj}$ is an iid array in $G$ such that the associated random walk converges in law to a L{\'e}vy process. Then a simple consequence of Theorem~\ref{thm_iidConvLevy} is that $\theta$ scales the array $X_{nj}$.
\end{remark}

\subsection{Approximating walk}\label{subsec:approxWalk}

In this subsection, given a L{\'e}vy process $\X$ in $G$, we construct an iid array $X_{nj}$ for which the associated random walk $\X^n$ converges in law to $\X$. The array $X_{nj}$ has the advantage that it takes values in either the support of the L{\'e}vy measure of $\X$, or in a set which shrinks to the identity as $n \rightarrow \infty$. This makes the walk $\X^n$ significantly easier to analyse than the increments of $\X$ itself and will be used in the proofs of Theorems~\ref{thm_LevyFinitepVar} and~\ref{thm_LKFormula}.

Throughout this subsection, let $\X$ be a L{\'e}vy process in $G$ with triplet $(A, B, \Pi)$. For $i\in \{1, \ldots m\}$ define
\[
\Gamma_i := \left\{ 0 \leq \gamma < 2 \midset \int_G |\xi_i(x)|^\gamma \Pi(dx) = \infty\right\}.
\]
Define also the sets of indexes
\[
J = \left\{j \in \{1, \ldots, m\} \midset A^{j,j} > 0\right\},
\]
\[
\widetilde K = \left\{k \in \{1, \ldots, m\} \midset 1 \notin \Gamma_k\right\}.
\]
For $k \in \widetilde K$ define
\[
\widetilde B^k = B^k - \int_{G} \xi_k(x) \Pi(dx),
\]
and let $K = \left\{ k \in \widetilde K \midset \widetilde B^k \neq 0\right\}$.

For $n$ sufficiently large so that $\Pi(U^c) < n/2$, let
\[
h_n = \inf\left\{h \geq 0 \midset \Pi\left(\{|\xi(x)| > h\} \cup U^c\right) \leq n/2\right\}.
\]
Define $U_n = \left\{x \in U \midset |\xi(x)| \leq h_n\right\}$ and note that $w_n := \Pi\{U_n^c\} \leq n/2$. Remark that $\lim_{n \rightarrow \infty} h_n = 0$ which implies that $U_n$ shrinks to $1_G$ as $n \rightarrow \infty$.

Define on $G$ the probability measure $\mu_n(dx) := w_n^{-1}\1{x \in U_n^c} \Pi(dx)$.
Observe that by H{\"o}lder's inequality, for all $q \geq 1$
\begin{equation}\label{eq_Holder}
\int_{U_n^c} |\xi_i(x)| \Pi(dx) \leq (n/2)^{1-1/q}\left(\int_{G}|\xi_i(x)|^q\Pi(dx)\right)^{1/q}.
\end{equation}

For every $n \geq 1$, let $Y_n = Y_n^1u_1 + \ldots Y_n^m u_m$ be a $\g$-valued random variable such that for all $k \in \widetilde K$
\[
b_n^k := \EEE{Y_n^k} = (1-w_n/n)^{-1}n^{-1} \widetilde B^k,
\]
and for all $k \notin \widetilde K$
\[
b_n^k := \EEE{Y_n^k} = (1-w_n/n)^{-1}n^{-1} \left(B^k - \int_{U_n^c}\xi_k(x)\Pi(dx)\right),
\]
and with covariances for all $i,j\in \{1,\ldots, m\}$
\[
\EEE{(Y_n^i - b_n^i)(Y_n^j-b_n^j)} = (1-w_n/n)^{-1}n^{-1}A^{i,j}.
\]
In particular, note that $Y_n^i = b_n^i$ a.s. for all $i \notin J$. Remark that setting $q = 2$ in~\eqref{eq_Holder} implies
\[
\lim_{n\rightarrow \infty} n^{-1} \int_{U_n^c} |\xi_i(x)| \Pi(dx) = 0,
\]
from which it follows that $\sup_{n \geq 1}n|b^i_n| < \infty$.
Moreover, it holds that
\[
\lim_{n \rightarrow \infty}\EEE{(Y_n^i - b_n^i)(Y_n^j-b_n^j)} = 0.
\]
It follows that we can choose $Y_n$ such that $\exp(Y_n)$ has support in a neighbourhood $V_n$ of $1_G$, such that $V_n$ shrinks to $1_G$ as $n \rightarrow \infty$. Denote by $\nu_n$ the probability measure of the $G$-valued random variable $\exp(Y_n)$.

Finally, let $X_{n1}$ be the $G$-valued random variable associated to the probability measure $(w_n/n) \mu_n + (1-w_n/n) \nu_n$, and let $X_{n2}, \ldots, X_{nn}$ be independent copies of $X_{n1}$.

Consider the random walk $\X^n$ associated with $X_{nj}$. Then a straightforward application of Theorem~\ref{thm_iidConvLevy} implies that $\X^n\convd \X$ as $D_o([0,1],G)$-valued random variables. We also record the following two simple lemmas whose proofs we omit.

\begin{lemma}\label{lem_thetaScales}
Let $0 < q_1, \ldots, q_m \leq 2$ be real numbers such that $q_i \notin \Gamma_i$ for all $i \in \{1,\ldots, m\}$, $q_i = 2$ for all $i \in J$, and $q_i \geq 1$ for all $i \in K$. Let $\theta$ be a scaling function such that $\theta(x) = \sum_{i = 1}^m |\xi_i(x)|^{q_i}$ for $x$ in a neighbourhood of $1_G$. Then $\theta$ scales the array $X_{n1},\ldots, X_{nn}$.
\end{lemma}

\begin{lemma}\label{lem_contOnSupport}
Let $\theta$ be a scaling function on $G$ which scales $X_{nj}$. Let $V$ be a neighbourhood of $1_G$, and let $f : \supp(\Pi) \cup V \mapsto \R$ be a bounded measurable function such that $f$ is continuous on $\supp(\Pi)$. Furthermore, suppose that
\[
\lim_{x \rightarrow 1_G} \frac{1}{\theta(x)} \norms{f - f(1_G) - \sum_{i=1}^m b_i\xi_i(x) - \frac{1}{2}\sum_{i,j=1}^m a_{i,j}\xi_i(x)\xi_j(x)} = 0.
\]

Then for all $n$ sufficiently large, $X_{n1} \in \supp(\Pi) \cup V$ a.s., and
\[
\lim_{n \rightarrow \infty} n\EEE{f(X_{n1}) - f(1_G)} = Q,
\]
where
\[
Q := \sum_{i=1}^m B^ib_i + \frac{1}{2}\sum_{i,j=1}^m A^{i,j}a_{i,j} + \int_{G} \left[ f(x) - f(1_G) - \sum_{i=1}^m b_i\xi_i(x) \right] \Pi(dx).
\]
\end{lemma}

\section{Homogeneous groups}
\label{sec:Hom}

In this section we collect several preliminary facts about homogeneous groups. For details, we refer to~\cite{FollandStein82} and~\cite{HebischSikora90}.

Throughout this section, we fix a homogeneous group $G$. That is, $G$ is a nilpotent, connected, and simply connected Lie group endowed with a one-parameter family of dilations (group automorphisms) $(\delta_\lambda)_{\lambda > 0}$, which, upon identifying $G$ with its Lie algebra $\g$ by the $\exp$ map, is given by
\[
\delta_\lambda(u_i) = \lambda^{d_i} u_i
\]
for a basis $u_1,\ldots, u_m$ of $\g$ and real numbers $d_m \geq \ldots \geq d_1 \geq 1$. We equip $G$ with a sub-additive homogeneous norm $\normc$ which induces a left-invariant metric $d(x,y) = \norm{x^{-1}y}$ (see~\cite{HebischSikora90}).

For the remainder of the section, we identify $G$ with $\g$ by the diffeomorphism $\exp : \g \mapsto G$, and write $x = \sum x^i u_i$ for $x \in G$.
Note that $\normt{x} = \sum_{i=1}^m |x^i|^{1/d_i}$ is also a homogeneous norm on $G$ and thus equivalent to $\normc$.

For a multi-index $\alpha = (\alpha^1,\ldots, \alpha^m)$, $\alpha^i \geq 0$, we define $\deg(\alpha) = \sum_{i=1}^m \alpha^i d_i$, and for $x \in G$, write $x^\alpha = (x^1)^{\alpha^1}\ldots(x^m)^{\alpha^m}$. By the Campbell--Baker--Hausdorff (CBH) formula, for all $i \in \{1,\ldots, m\}$ there exist constants $C^i_{\alpha,\beta}$ such that
\begin{equation}\label{eq:CBH}
(xy)^i = x^i + y^i + \sum_{\alpha,\beta} C^i_{\alpha,\beta} x^\alpha y^\beta,
\end{equation}
where the (finite) sum runs over all non-zero multi-indexes $\alpha,\beta$ such that $\deg(\alpha) + \deg(\beta) = d_i$.

\begin{example}\label{ex:gradedLie}
Recall that a Lie group $G$ is called graded if its Lie algebra is endowed with a decomposition
\begin{equation}\label{eq:LieAlgDecomp}
\g = \g^1 \oplus \ldots \oplus \g^N
\end{equation}
such that $[\g^i,\g^j] \subseteq \g^{i+j}$, where $\g^k = 0$ for $k > N$ (and where we allow the possibility that $\g^k = 0$ for some $k \leq N$). Every graded Lie group can be equipped with a natural family of dilations $(\delta_{\lambda})_{\lambda > 0}$, and thus a homogeneous structure, for which $d_1,\ldots, d_m$ are rational numbers with $d_1 = 1$, given by $\delta_\lambda(u) = \lambda^{k/\alpha}u$ for all $u \in \g^k$, where $\alpha = \min\{k \geq 1 \mid \g^k \neq 0\}$ (and conversely, if $d_1,\ldots, d_m$ are rational for a homogeneous group $G$, then $G$ can be given a graded structure~\cite[p. 5]{FollandStein82}).

Recall also that a graded Lie group $G$ is called a step-$N$ Carnot group (or stratified group in the terminology of~\cite{FollandStein82}) if the decomposition~\eqref{eq:LieAlgDecomp} further satisfies $[\g_i,\g_j] = \g_{i+j}$, where $\g_k = 0$ for $k > N$. Every Carnot group is a homogeneous group with a natural family of dilations given by $\delta_\lambda(u) = \lambda^k u$ for all $u \in \g_k$ (so that $d_i \in \{1,\ldots, N\}$), and for which the metric $d$ can be taken as the Carnot--Carath{\'e}odory distance~\cite[p. 38]{Baudoin04}.

The Carnot group which will be particularly relevant in Section~\ref{subsec:RPs} for applications in rough paths theory is the step-$N$ free nilpotent Lie group $G^N(\R^d)$ over $\R^d$, which we recall is, by definition, the space where geometric $p$-rough paths (for $\floor p = N$) take value. For further details concerning the theory of geometric rough paths, we refer to~\cite{FrizVictoir10}.
\end{example}

\begin{remark}\label{remark:branchedRPs}
Another homogeneous group which plays an important role in the theory of rough paths is the step-$N$ Butcher group $\GG^N(\R^d)$ over $\R^d$ (see~\cite{Gubinelli10, HairerKelly15}). Recall that $G^N(\R^d)$ is canonically embedded in $\GG^N(\R^d)$, and that $\GG^N(\R^d)$ admits a natural grading under which $\GG^N(\R^d)$ is not a Carnot group (see~\cite[Remark~2.15]{HairerKelly15}).

The group $\GG^N(\R^d)$ is, by definition, the space where branched rough paths take value (which form a genuine extension of the notion of geometric rough paths). 
We mention that branched rough paths were recently studied in~\cite{BCFP17} to give a rough path perspective on renormalisation of stochastic PDEs in the theory of regularity structures~\cite{BHZ16, Hairer14}. L{\'e}vy processes in $\GG^N(\R^d)$ in particular form a family of stationary stochastic processes closed under appropriate renormalisation maps (see~\cite[Section~4]{BCFP17}).
\end{remark}

\subsection{Paths of finite $p$-variation}

For $p > 0$ and functions $\x,\y : [s,t] \mapsto G$, define the $p$-variation distance
\[
d_{\pvar;[s,t]}(\x,\y) = d(\x_s,\y_s) + \sup_{\DD\subset [s,t]} \left( \sum_{t_j \in \DD} d(\x_{t_j,t_{j+1}}, \y_{t_j,t_{j+1}})^p \right)^{1/p},
\]
where the supremum runs over all partitions $\DD$ of $[s,t]$ and where we have used the shorthand notation $\x_{u,v} = \x_u^{-1} \x_v$. Define also the $p$-variation of $\x$ by $\norm{\x}_{\pvar;[s,t]} = d_{\pvar;[s,t]}(\x,1_G)$, and
\[
d_{0;[s,t]}(\x,\y) = d(\x_s,\y_s) + \sup_{u,v \in [s,t]} d(\x_{u,v},\y_{u,v})
\]
and
\[
d_{\infty;[s,t]}(\x,\y) = \sup_{u \in [s,t]} d(\x_u,\y_u), \; \; \norm{\x}_{\infty;[s,t]} = d_{\infty;[s,t]}(\x,1_G).
\]
We will drop the reference to the interval $[s,t]$ when it is clear from the context. For convenience, we record the following standard interpolation estimates.

\begin{lemma}\label{lem:interpolate}
\begin{enumerate}[label={(\arabic*)}]
\item \label{point:I1} For all $p' > p > 0$ and $\x,\y : [s,t] \mapsto G$
\[
d_{\pprimevar}(\x,\y) \leq 2^{\max\{0,(1-p)/p'\}}(\norm{\x}_{\pvar} + \norm{\y}_{\pvar})^{p/p'} d_{0}(\x,\y)^{1-p/p'}.
\]

\item \label{point:I2} There exists $C > 0$ such that for all $\x,\y : [s,t] \mapsto G$ with $\x_s = \y_s$
\[
d_{\infty}(\x,\y) \leq
d_{0}(\x,\y) \leq C \max\{d_{\infty}(\x,\y), d_{\infty}(\x,\y)^{1/d_m}(\norm{\x}_{\infty} + \norm{\y}_{\infty})^{1-1/d_m}\}.
\]
\end{enumerate}
\end{lemma}

\begin{proof}
\ref{point:I1} is obvious.
To show~\ref{point:I2}, it follows from an application of the CBH formula~\eqref{eq:CBH} and the equivalence of $\normc$ and $\normt{\cdot}$, that for all $g,h \in G$
\[
\norm{g^{-1}hg} \leq C_1 \max\{\norm{h}, \norm{h}^{1/d_m}\norm{g}^{1-1/d_m}\}.
\]
The conclusion now follows by the identical argument used to prove~\cite[Proposition~8.15]{FrizVictoir10}.
\end{proof}

For $p \geq 1$, let $C^{\pvar}([0,T],G)$ denote the space of continuous paths of finite $p$-variation equipped with the metric $d_{\pvar;[0,T]}$. Note that $C^{\pvar}([0,T],G)$ is a complete metric space due to the lower semi-continuity of $\x \mapsto \norm{\x}_{\pvar;[0,T]}$ (under pointwise convergence).

Note that, except in trivial cases, $C^{\pvar}([0,T],G)$ is non-separable. However, it is not difficult to show that $C^{\pprimevar}([0,T],G)$ contains a separable subset $C^{0,\pprimevar}([0,T],G)$ which contains $C^{\pvar}([0,T],G)$ for all $1 \leq p < p'$.
Indeed, let $C^g([0,T],G)$ denote the space of curves which are concatenations of one-parameter subgroups of $G$, i.e., all curves $\gamma : [0,T] \mapsto G$ of the form
\begin{equation}\label{eq:piecewise}
\gamma(t) = \gamma(t_{k-1}) \exp\left(\frac{t-t_{k-1}}{t_{k} - t_{k-1}} \log x_k \right), \; \; t \in [t_{k-1},t_{k}], \; \; k \in \{1, \ldots, n\},
\end{equation}
where $\DD = (t_0 = 0 < t_1 < \ldots < t_n = T)$ is a partition of $[0,T]$ and $x_1,\ldots, x_n \in G$ (and where for clarity we have broken the convention of identifying $G$ with $\g$). Then for $p \geq 1$, define $C^{0,\pvar}([0,T],G)$ as the closure of $C^{g}([0,T],G) \cap C^{\pvar}([0,T],G)$ in $C^{\pvar}([0,T],G)$.

\begin{remark}
In the case that $G$ is a Carnot group with decomposition~\eqref{eq:LieAlgDecomp}, $C^{0,\pvar}([0,T],G)$ is precisely the closure of the horizontal lifts of smooth paths $\gamma \in C^\infty([0,T], \g^1)$.
\end{remark}

To show the claimed properties of $C^{0,\pvar}([0,T],G)$, note that for $x \in G$, the path $\gamma : t \mapsto \exp(t\log x)$ has finite $p$-variation if and only if $x^i = 0$ for all $i \in \{1,\ldots, m\}$ such that $d_i > p$, in which case there exists $C_1 = C_1(p,G)>0$ such that $\norm{\gamma}_{\pvar;[0,1]} \leq C_1 \norm{x}$. For $\x : [0,T] \mapsto G$, and a partition $\DD \subset [0,T]$, let $\x^\DD \in C^g([0,T], G)$ be the interpolation of $\x$ along $\DD$ defined as $\gamma$ in~\eqref{eq:piecewise} with $x_k = \x^{-1}_{t_{k-1}}\x_{t_k}$ and $\x^\DD_0 = \x_0$. 
One can then readily show (e.g., by Lemma~\ref{lem:xpvarBound}) that $\sup_{\DD \subset [0,T]}\norm{\x^\DD}_{\pvar} \leq C_2\norm{\x}_{\pvar}$. Hence for all $\x \in C^{\pvar}([0,T],G)$ and $p' > p \geq 1$, by Lemma~\ref{lem:interpolate}, $d_{\pprimevar;[0,T]}(\x^\DD,\x) \rightarrow 0$ as $|\DD| \rightarrow 0$, which shows that $C^{\pvar}([0,T],G) \subseteq C^{0,\pprimevar}([0,T],G)$ as claimed. The fact that $C^{0,\pvar}([0,T],G)$ is separable (and thus Polish) is also easy to show (e.g., by considering $\gamma \in C^g([0,T],G)$ with rational coordinates and using a similar argument as the proof of Lemma~\ref{lem:xpvarBound}).

The following result will be important in our classification of $G$-valued L{\'e}vy processes of finite $p$-variation.

\begin{proposition}\label{prop:finiteVarLim}
Let $p > 0$ and $(\X_n)_{n \geq 1}$ be a sequence of $D([0,T], G)$-valued random variables such that $(\norm{\X_n}_{\pvar;[0,T]})_{n \geq 1}$ is a tight collection of real random variables. Suppose that $\X_n \convd \X$ as $D([0,T], G)$-valued random variables.

\begin{enumerate}[label={(\arabic*)}]
\item \label{point:finitepVar} It holds that $\norm{\X}_{\pvar;[0,T]} < \infty$ a.s..

\item \label{point:pVarConver} Suppose further that $p \geq 1$ and $\X^n, \X$ are $C([0,T], G)$-valued random variables. Then for all $p' > p$, $\X_n \convd \X$ as $C^{0,\pprimevar}([0,T], G)$-valued random variables.
\end{enumerate}
\end{proposition}

\begin{proof}
\ref{point:finitepVar} Note that $\x \mapsto \norm{\x}_{\pvar}$ is a lower semi-continuous function on $D$. Since $D([0,T], G)$ is Polish, we may apply the Skorokhod representation theorem~\cite[Theorem~3.30]{Kallenberg97}, from which the conclusion easily follows.

\ref{point:pVarConver} It follows from Lemma~\ref{lem:interpolate} that every set of the form $A \cap \{\x \in C([0,T],G) \mid \norm{\x}_{\pvar;[0,T]} < R\}$, where $R > 0$ and $A$ is a compact subset of $C([0,T],G)$ (for uniform topology), is a compact subset of $C^{0,\pprimevar}([0,T],G)$. Hence $(\X_n)_{n \geq 1}$ is a tight collection of $C^{0,\pprimevar}([0,T],G)$-valued r.v.'s, and so converges in law along a subsequence to some $C^{0,\pprimevar}([0,T],G)$-valued r.v. $\widetilde \X$.
Since $\X_n \convd \X$ as $C([0,T], G)$-valued r.v.'s, it necessarily follows that $\widetilde \X \eqd \X$, which concludes the proof.
\end{proof}

\begin{remark}
A version of Helly's selection principle (see~\cite[Theorem~2.4]{Porter05}) states that any uniformly bounded sequence of functions $\x^n : [0,T] \mapsto G$ for which $\sup_{n \geq 1} \norm{\x^n}_{\pvar;[0,T]} < \infty$ for some $p \geq 1$, has a subsequence such that $\x^{n_k} \rightarrow \x$ pointwise.
\end{remark}

\section{$p$-variation tightness of random walks}\label{sec:tightRandWalks}

We continue to use the notation of the previous section. Consider an iid array $X_{nk}$ in the homogeneous group $G$, and let $\X^n$ be the associated random walk. The main result of this section is Theorem~\ref{thm_cadlagPathTight}, which provides sufficient conditions under which $(\norm{\X^n}_{\pvar;[0,1]})_{n \geq 1}$ is tight. In its simplest form, Theorem~\ref{thm_cadlagPathTight} implies that whenever $\X^n$ converges in law to a L{\'e}vy process in $G$, and the array $X_{nk}$ is scaled by a scaling function $\theta$, then $(\norm{\X^n}_{\pvar;[0,1]})_{n \geq 1}$ is tight for all $p > \kappa > 0$, where $\kappa$ depends only on the scaling function $\theta$.

Let $\xi_1,\ldots,\xi_m \in C^\infty_c(G)$ and $\xi : G \mapsto \g$ be smooth functions and $U$ a neighbourhood of $1_G$ for which the conditions at the start of Section~\ref{sec:LieGroupIidArrays} are satisfied with respect to the basis $u_1,\ldots, u_m$.

\begin{theorem}\label{thm_cadlagPathTight}
Let $X_{n1}, \ldots, X_{nn}$ be an iid array of $G$-valued random variables and $\X^n$ the associated random walk. For every $i\in \{1,\ldots, m\}$, let $0 < q_i \leq 2$ be a real number, and define
\[
\kappa = \max\{q_1 d_1,\ldots,q_m d_m\}.
\]
Consider the following conditions:
\begin{enumerate}[label={(\Alph*)}]
\item\label{point_hTight} for every fixed $h \in [0,1]$, $(\X^n_h)_{n \geq 1}$ is a tight collection of $G$-valued random variables;

\item\label{point_nEXiBound} for all $i \in \{1,\ldots, m\}$, $\sup_{n \geq 1} n\norms{\EEE{\xi_i(X_{n1})}} < \infty$;

\item\label{point_thetaScale} the array $X_{nk}$ is scaled by a scaling function $\theta$, where $\theta \equiv \sum_{i=1}^m |\xi_i|^{q_i}$ on a neighbourhood of $1_G$.
\end{enumerate}

Then, provided~\ref{point_hTight},~\ref{point_nEXiBound} and~\ref{point_thetaScale} hold, $(\X^n)_{n \geq 1}$ is a tight collection of $D_o([0,1],G)$-valued random variables and, for every $p > \kappa$, $(\norm{\X^n}_{\pvar;[0,1]})_{n \geq 1}$ is a tight collection of real random variables.
\end{theorem}

\begin{remark}\label{remark:convLevyImpliesConds}
Suppose that for a L{\'e}vy process $\X$ in $G$, $\X^n \convd \X$ as $D_o([0,T], G)$-valued random variables. Then conditions~\ref{point_hTight} and~\ref{point_nEXiBound} are automatically satisfied by Theorem~\ref{thm_iidConvLevy} (and~\ref{point_thetaScale} is satisfied upon choosing $q_i = 2$ for all $i\in \{1,\ldots, m\}$ by Remark~\ref{remark:squaresScale}). 
\end{remark}

The remainder of the section is devoted to the proof of Theorem~\ref{thm_cadlagPathTight}, which can be split into three parts. The first part is collected in Section~\ref{subsec:pVarTight} and comprises a general $p$-variation tightness criterion for strong Markov processes.
The second part, which is the most technical part of the proof, is collected in Section~\ref{subsec:pMore} and establishes the bounds required to apply the results of Section~\ref{subsec:pVarTight} for the case $p > d_m$.
The third part is collected in Section~\ref{subsec:pLess} and treats the case $p \leq d_m$. Roughly speaking, in the third part we decompose $\X^n$ into the lift of a walk in a lower level group, for which the previous two parts apply,
and a perturbation on the higher levels, for which the $p$-variation can be controlled directly.

\subsection{$p$-variation tightness of strong Markov processes}\label{subsec:pVarTight}

In this section we give a criterion for $p$-variation tightness of strong Markov processes in a Polish space (Theorem~\ref{thm:pVarTightCrit}), which is inspired by the work of Manstavi{\v c}ius~\cite{Manstavicius04}.

Let $(E,d)$ be a metric space and $\x : [0,T] \mapsto E$ a function. Define
\[
M(\x) := \sup_{s,t \in [0,T]} d(\x_t, \x_s),
\]
and, for $\delta > 0$,
\[
\nu_\delta(\x) := \sup\left\{k \geq 0 \midset \exists (t_i)_{i=1}^{2k}, t_1 < t_2 \leq t_3 < \ldots < t_{2k}, d(\x_{t_{2i}}, \x_{t_{2i-1}}) > \delta, i \leq k\right\}.
\]
Note that quantity $\nu_\delta(\x)$ measures the maximum number of oscillations of $\x$ of magnitude greater than $\delta$ over non-overlapping intervals. Observe the following basic inequality which serves to control $\norm{\x}_{\pvar;[0,T]}$:
\begin{equation}\label{eq:pVarBoundNuM}
\norm{\x}^p_{\pvar;[0,T]} \leq \sum_{r = 1}^\infty 2^{-rp+p}\nu_{2^{-r}}(\x) + M(\x)^p\nu_{1}(\x).
\end{equation}

For $\delta > 0$, define the increasing sequence of times $(\tau^\delta_j(\x))_{j=0}^\infty$ by $\tau^\delta_0(\x) = 0$ and for $j \geq 1$
\[
\tau^\delta_j(\x) =
\begin{cases}
\inf \left\{t \in [\tau^\delta_{j-1}(\x), T] \midset \sup_{u,v \in [\tau^\delta_{j-1}(\x),t]} d(\x_u,\x_v) > \delta\right\} \\ 
\infty \; \; \; \mbox{ if the above set is empty}.
\end{cases}
\]

\begin{lemma}\label{lem:negBinomBound}
Let $\X$ be a $D([0,T],E)$-valued random variable for a Polish space $(E,d)$. Let $\delta,h > 0$ such that there exists $q \in (0,1)$ for which a.s. for all $i \geq 0$
\begin{equation}\label{eq:successiveTau}
\PPP{\tau^\delta_{i+1}(\X) - \tau^\delta_i(\X)  \leq h | \tau^\delta_{i}(\X),\ldots, \tau^\delta_{0}(\X)} \leq q
\end{equation}
(where we use the convention $\infty - \infty = \infty$). Then
\[
\EEE{\nu_\delta(\X)} \leq \roof{T/h}\frac{1}{1-q}.
\]
\end{lemma}

\begin{proof}
Note that for any function $\x : [0,T] \mapsto E$, it holds that $\nu_{\delta}(\x)$ is the largest integer $j$ for which $\tau^\delta_j(\x) \leq T$, and thus
\begin{equation*}
\PPP{\nu_{\delta}(\X) \geq j} = \PPP{\tau^{\delta}_j(\X) \leq T}.
\end{equation*}
For $i \geq 0$, consider the event $A_i = \{\tau_{i+1}^{\delta}(\X) - \tau_{i}^\delta(\X) > h\}$, and note that
\[
\PPP{\tau_j^\delta(\X) \leq T} \leq \PPP{\textnormal{at most $\roof{T/h}$ of $(A_i)_{i=0}^{j-1}$ occur}}.
\]
Consider a real random variable $Z$ distributed by the negative binomial distribution with parameters $(\roof{T/h},q)$, i.e., $Z$ counts the total number of iid Bernoulli trials with success probability $q$ until exactly $\roof{T/h}$ failures occur. It follows from the uniform bound~\eqref{eq:successiveTau} that
\[
\PPP{\textnormal{at most $\roof{T/h}$ of $(A_i)_{i=0}^{j-1}$ occur}} \leq \PPP{Z \geq j}
\]
(where one considers $A_i$ as a failure with probability at least $1-q$), so that
\[
\EEE{\nu_\delta(\X)} \leq \EEE{Z} = \roof{T/h}\frac{1}{1-q}.
\]
\end{proof}

We now show how one can verify the condition of Lemma~\ref{lem:negBinomBound} for a strong Markov process. We first restrict attention to the set of times on which a process is allowed to move.

\begin{definition}
For a metric space $(E,d)$ and a $D([0,T], E)$-valued random variable $\X$, call a (deterministic) open interval $(s,t) \subset [0,T]$ \emph{stationary} if
\[
\PPP{\forall u \in (s,t), \X_u = \X_s} = 1.
\]
Let $Z_\X\subseteq [0,T]$ denote the union of all stationary intervals, and let $R_{\X} = [0,T]\setminus Z_\X$ be its complement.
\end{definition}

\begin{example}
For the random walk $\X^n \in D([0,1], G)$ associated with an iid array $X_{nj}$ in a Lie group $G$, we have $R_{\X^n} = \{0,1/n,\ldots, (n-1)/n, 1\}$.
\end{example}

We emphasise that the role of $R_\X$ is only technical in that it allows us to easily formulate bounds uniform in $s \in R_\X$ (such as those in Theorem~\ref{thm:pVarTightCrit} and Corollary~\ref{cor:pVarTightCrit}) which hold for random walks and for which the same bounds would not hold when taken uniformly over all $s \in [0,T]$ (though for completely harmless reasons). The following lemma is a variant of G{\={\i}}hman--Skorokhod~\cite[Lemma~2, p. 420]{GihmanSkorohod74} (in which the notion of $R_\X$ does not appear).

\begin{lemma}[Maximum inequality]\label{lem:maximumIneq}
Let $\X$ be a c{\`a}dl{\`a}g (not necessarily strong) Markov process taking values in a Polish space $(E,d)$.

Let $h,\delta > 0$ and suppose there exists $c \in [0,1)$ such that
\begin{equation}\label{eq:alphaBound}
\sup_{s \in R_\X} \sup_{x \in E} \sup_{t \in [s,s+h]}\PPPover{s,x}{d(\X_s,\X_t) > \delta} \leq c.
\end{equation}

Then for all $s \in R_\X$ and $x \in E$, it holds that
\[
\PPPover{s,x}{\sup_{t \in [s,s+h]} d(\X_s,\X_t) > 2\delta} \leq \frac{\PPPover{s,x}{d(\X_s,\X_{s+h}) > \delta}}{1-c}.
\]
\end{lemma}

\begin{proof}
Let $s \in R_\X$ and observe that a.s.
\[
\sup_{t \in [s,s+h]} d(\X_s,\X_t) = \sup_{t \in [s,s+h] \cap R_\X} d(\X_s,\X_t).
\]

Consider a nested sequence of partitions $\DD_n \subset [s,s+h] \cap R_\X$ such that
\[
\lim_{n \rightarrow\infty}\sup_{t \in [s,s+h]\cap R_\X} d_+(t,\DD_n) = 0,
\]
where $d_+(t, \DD_n) := t - \sup\left\{u \leq t\midset u\in \DD_n\right\}$.
Since $\X$ is c{\`a}dl{\`a}g, it holds that
\[
\sup_{t \in [s,s+h] \cap R_\X} d(\X_s,\X_t) = \lim_{n \rightarrow \infty} \sup_{t_i \in \DD_n} d(\X_s,\X_{t_i}),
\]
where the right side is non-decreasing in $n$ since $\DD_n$ are nested.

It thus suffices to show that for any partition $\DD = (t_0=s,\ldots, t_n) \subset [s,s+h]\cap R_\X$, we have
\begin{equation}\label{eq:maxOverPart}
\PPPover{s,x}{\sup_{t_i \in \DD} d(\X_s,\X_{t_i}) > 2\delta} \leq \frac{\PPPover{s,x}{d(\X_s,\X_{s+h}) > \delta}}{1-c}.
\end{equation}

To this end, for $i \in \{0,\ldots, n\}$, consider the events
\[
C_i := \{d(\X_{t_i},\X_{s+h}) > \delta\}
\]
and
\[
B_i := \{d(\X_s,\X_{t_i}) > 2\delta\}.
\]
Define the $\sigma$-algebras $\FF_{s,t} := \sigma(\X_u)_{s\leq u \leq t}$. Observe that~\eqref{eq:alphaBound} implies that a.s.
\[
\EEEover{s,x}{ \1{C_i} \midset \FF_{s,t_i}} \leq c.
\]
Moreover, consider the disjoint events $F_i := B^c_1 \cap \ldots \cap B^c_{i-1} \cap B_{i}$.
Then for all $i \in \{0,\ldots, n\}$
\[
F_i \cap C^c_{i} \subseteq C_0.
\]
Since each $F_i\cap C^c_i$ is disjoint and $F_i$ is $\FF_{s,t_i}$-measurable, we have
\begin{align*}
\PPPover{s,x}{C_0}
&\geq \sum_{i=1}^n \PPPover{s,x}{F_i \cap C^c_{i}} \\
&\geq \sum_{i=1}^n \EEEover{s,x}{\1{F_i}\EEEover{s,x}{ \1{C^c_i} \midset \FF_{s,t_i}}} \\
&\geq (1-c)\sum_{i=1}^n \PPPover{s,x}{F_i}.
\end{align*}
Finally,~\eqref{eq:maxOverPart} now follows from the fact that
\[
\sum_{i=1}^n \PPPover{s,x}{F_i} = \PPPover{s,x}{\sup_{t_i \in \DD} d(\X_s,\X_{t_i}) > 2\delta}.
\]
\end{proof}

\begin{corollary}\label{cor:tauGapBound}
Let $\X$ be a c{\`a}dl{\`a}g strong Markov process taking values in a Polish space $(E,d)$. Let $h,\delta > 0$ and $c \in [0,1)$ satisfy~\eqref{eq:alphaBound}. Then for all $i \geq 0$, a.s.
\[
\PPP{\tau^{4\delta}_{i+1}(\X) - \tau^{4\delta}_{i}(\X) \leq h | \tau^{4\delta}_i(\X),\ldots,\tau^{4\delta}_0(\X)} \leq \frac{c}{1-c}.
\]
\end{corollary}

\begin{proof}
Observe that $\tau^{4\delta}_i(\X)$ takes values a.s. in $R_\X$, and that the event $\{\tau_{i+1}^{4\delta}(\X) - \tau_i^{4\delta}(\X) \leq h\}$ is contained inside $\{\sup_{t \in [\tau_i,\tau_i + h]} d(\X_{\tau_i},\X_{t}) > 2\delta\}$. Conditioning on the stopping times $\{\tau^{4\delta}_i(\X),\ldots,\tau^{4\delta}_0(\X)\}$ and using the assumption that $\X$ is a strong Markov process, the desired result now follows from Lemma~\ref{lem:maximumIneq}.
\end{proof}

We now obtain the following $p$-variation tightness criterion for strong Markov processes. Recall the quantity $M(\X) = \sup_{s,t \in [0,T]} d(\X_t, \X_s)$.

\begin{theorem}\label{thm:pVarTightCrit}
Let $\MM$ be collection of c{\`a}dl{\`a}g strong Markov processes on $[0,T]$ taking values in a Polish space $(E,d)$.
Suppose that
\begin{enumerate}[label={(\alph*)}]
\item\label{point:MnuTight} $(M(\X))_{\X \in \MM}$ is tight, and

\item\label{point:supsBound} there exist constants $a, \kappa, b > 0$ and $c \in [0,1/2)$ such that for all $\delta \in (0, b]$
\[
\sup_{\X \in \MM} \sup_{s \in R_{\X}} \sup_{x \in E} \sup_{t \in [s,s+h(\delta)]} \PPPover{s,x}{d(\X_s,\X_{t}) > \delta} \leq c,
\]
where $h(\delta) := a\delta^\kappa$.
\end{enumerate}
Then for any $p > \kappa$, $(\norm{\X}_{\pvar;[0,T]})_{\X\in \MM}$ is a tight collection of real random variables.
\end{theorem}

\begin{proof}
Let $p > \kappa$. We claim that it suffices to show
\begin{equation}\label{eq:nuSeriesConv}
\sup_{\X \in \MM} \sum_{r = 0}^\infty 2^{-rp}\EEE{\nu_{2^{-r}}(\X)} < \infty.
\end{equation}
Indeed, observe that~\eqref{eq:nuSeriesConv} implies that $(\nu_{1}(\X))_{\X\in\MM}$ is tight. It then follows, by~\ref{point:MnuTight} and the estimate~\eqref{eq:pVarBoundNuM}, that~\eqref{eq:nuSeriesConv}  implies $(\norm{\X}_{\pvar;[0,T]})_{\X\in \MM}$ is tight as claimed.

It thus remains to show~\eqref{eq:nuSeriesConv}. By~\ref{point:supsBound} and Corollary~\ref{cor:tauGapBound}, it holds that for all $\delta \in (0,b]$
\[
\sup_{\X \in \MM}\PPP{\tau_{i+1}^{4\delta}(\X) - \tau_i^{4\delta}(\X)  \leq h(\delta)\midset \tau_i^{4\delta}(\X),\ldots, \tau^{4\delta}_0(\X)} \leq \frac{c}{1-c}.
\]

Hence, by Lemma~\ref{lem:negBinomBound}, for all $\delta \in (0,b]$
\[
\sup_{\X \in \MM}\EEE{\nu_{4\delta}(\X)} \leq \roof{T/h(\delta)}\frac{1}{1-\frac{c}{1-c}} \leq (1 + T \delta^{-\kappa}/a)\frac{1-c}{1-2c},
\]
from which~\eqref{eq:nuSeriesConv} readily follows.
\end{proof}

\begin{corollary}\label{cor:pVarTightCrit}
Let $(\X^n)_{n \geq 1}$ be a sequence of c{\`a}dl{\`a}g strong Markov processes on $[0,T]$ taking values in a Polish space $(E,d)$. Suppose that

\begin{enumerate}[label={(\roman*)}]
\item\label{point:coordsTight} for every fixed rational $h \in [0,T]$, $(\X_h^n)_{n \geq 1}$ is a tight collection of $E$-valued random variables, and
\item\label{point:betaGammaBound} there exist constants $K, \beta, \gamma, b > 0$ such that for all $\delta \in (0, b]$ and $h > 0$
\[
\sup_{n \geq 1} \sup_{s \in R_{\X^n}} \sup_{x \in E} \sup_{t \in [s,s+h]} \PPPover{s,x}{d(\X^n_s,\X^n_t) > \delta} \leq K\frac{h^\beta}{\delta^\gamma}.
\]
\end{enumerate}
Then $(\X^n)_{n \geq 1}$ is a tight collection of $D([0,T],E)$-valued random variables, and for any $p > \gamma/\beta$, $(\norm{\X^n}_{\pvar;[0,T]})_{n \geq 1}$ is a tight collection of real random variables.
\end{corollary}

\begin{proof}
First, note that~\ref{point:betaGammaBound} applied to small $h$ allows us to verify the Aldous condition for the sequence $(\X^n)_{n \geq 1}$ (see, e.g,~\cite[p. 188]{Kolokoltsov11}, though note one should restrict attention to sequences of stopping times $\tau_n$ taking values in $R_{\X^n}$ a.s., which is a trivial modification to the usual Aldous condition). Together with~\ref{point:coordsTight}, it follows that $(\X^n)_{n \geq 1}$ is a tight collection of $D([0,T],E)$-valued random variables (\cite[Theorems~4.8.1,~4.8.2]{Kolokoltsov11}).

Observe that $M$ is a continuous function on $D([0,T],E)$, from which it follows that $(M(\X^n))_{n \geq 1}$ is tight. Moreover, observe that~\ref{point:betaGammaBound} implies that there exists $a > 0$ such that for all $\delta \in (0,b]$
\[
\sup_{n \geq 1} \sup_{s \in R_{\X^n}} \sup_{x \in E} \sup_{t \in [s,s+h]} \PPPover{s,x}{d(\X^n_s,\X^n_t) > \delta} \leq \frac{1}{3},
\]
where $h = a\delta^{\gamma/\beta}$. It follows that the conditions of Theorem~\ref{thm:pVarTightCrit} are satisfied with $\kappa = \gamma/\beta$, so that indeed $(\norm{\X^n}_{\pvar;[0,T]})_{n \geq 1}$ is tight for all $p > \gamma/\beta$.
\end{proof}

\subsection{Proof of Theorem~\ref{thm_cadlagPathTight} in the case $p > d_m$} \label{subsec:pMore}

We continue using the notation of Section~\ref{sec:Hom}. In particular, we identify $G$ with $\g$ via the $\exp$ map.

\begin{remark}\label{remark:pBiggerThanN}
We note here that Corollary~\ref{cor:pVarTightCrit} and the bound~\eqref{eq_XnUnifBound} in the upcoming Lemma~\ref{lem:boundsForTight} are sufficient to establish that conditions~\ref{point_hTight}, \ref{point_nEXiBound} and \ref{point_thetaScale} imply that $(\X^n)_{n \geq 1}$ is a tight collection of $D_o([0,1],G)$-valued random variables and that $(\norm{\X^n}_{\pvar;[0,1]})_{n \geq 1}$ is tight for all $p > \kappa \vee d_m$, which proves the statement of Theorem~\ref{thm_cadlagPathTight} subject to the restriction $p > d_m$.
\end{remark}

Observe that an inductive application of the CBH formula~\eqref{eq:CBH}, along with the multinomial identity $(z_1+\ldots + z_n)^{j} = \sum_{k_1+\ldots + k_n = j} \binom{j}{k_1,\ldots,k_n} z_1^{k_1}\ldots z_n^{k_n}$, yields the following lemma.

\begin{lemma}\label{lem:CBH}
For all $x_1,\ldots, x_k \in G$ and every index $i \in \{1,\ldots, m\}$, it holds that
\[
(x_1\ldots x_k)^i = \sum_{1\leq a_1 \leq k} x^i_{a_1} + \sum_{r=2}^{\floor{d_i}}\sum_{\alpha_1,\ldots, \alpha_r}\sum_{1 \leq a_1 < \ldots < a_r \leq k} c^i_{\alpha_1,\ldots,\alpha_r} x^{\alpha_1}_{a_1}\ldots x^{\alpha_r}_{a_r},
\]
where $\sum_{\alpha_1,\ldots, \alpha_r}$ indicates the (finite) sum over all non-zero multi-indexes $\alpha_1,\ldots, \alpha_r$ such that $\deg(\alpha_1) + \ldots + \deg(\alpha_r) = d_i$ and $c^i_{\alpha_1,\ldots,\alpha_r}$ are constants.
\end{lemma}

Recall that $\X^n$ denotes the random walk associated to the iid array $X_{n1},\ldots, X_{nn}$.

\begin{lemma}\label{lem:boundsForTight}
Use the notation from Theorem~\ref{thm_cadlagPathTight} and suppose that~\ref{point_nEXiBound} and~\ref{point_thetaScale} hold. Let $\gamma := d_m \vee \kappa$, and for $i \in \{1,\ldots, m\}$ denote by $\Y^{n,i} \in D_o([0,1],\R)$ the random walk associated with the $\R$-valued iid array $X^i_{nk}$.

Then there exists $K > 0$ such that for all $n \geq 1$, $k \in \{1,\ldots, n\}$ and $\delta \in (0,1]$
\begin{equation}\label{eq_XnUnifBound}
\PPP{\norm{\X^n_{k/n}} > \delta} \leq K\frac{k/n}{\delta^\gamma},
\end{equation}
and, for all $i \in \{1,\ldots, m\}$ such that $q_i \leq 1$,
\begin{equation}\label{eq_YnUnifBound}
\PPP{\norms{\Y^{n,i}_{k/n}} > \delta} \leq K\frac{k/n}{\delta^{q_i}}.
\end{equation}
\end{lemma}

\begin{proof}
We first claim that it suffices to consider the case $\norm{X_{n1}} \leq \varepsilon$ a.s. for all $n \geq 1$, where $\varepsilon > 0$ may be taken arbitrarily small. Indeed, let $\varepsilon > 0$ and note that there exists $c > 0$ such that $\theta(x) > c$ for all $x \in G$ with $\norm{x} > \varepsilon$. Since $\theta$ scales $X_{nk}$, it follows that there exists $C_1 > 0$ such that for all $n \geq 1$
\[
\PPP{\norm{X_{n1}} \geq \varepsilon} \leq c^{-1}\EEE{\theta(X_{n1})} \leq C_1/n,
\]
and hence
\[
\PPP{\max_{1 \leq a \leq k}\norm{X_{na}} \geq \varepsilon} = 1 - \left(1-\PPP{\norm{X_{n1}} \geq \varepsilon}\right)^k \leq C_1k/n.
\]
It follows that for all $n \geq 1$ and $k \in \{1,\ldots, n\}$
\[
\PPP{\norm{\X^n_{k/n}} > \delta} \leq \PPP{\norm{\X^n_{k/n}} > \delta, \max_{1 \leq a \leq k}\norm{X_{na}} < \varepsilon} + C_1k/n,
\]
and similarly for $\PPP{\norms{\Y^{n,i}_{k/n}} > \delta}$.
Replacing $X_{nk}$ by
\[
X_{nk}' =
\begin{cases} X_{nk} &\mbox{if } \norm{X_{nk}} < \varepsilon \\ 
1_G &\mbox{otherwise},
\end{cases}
\]
we note that~\ref{point_nEXiBound} and~\ref{point_thetaScale} imply that the same conditions hold for the iid array $X_{nk}'$. It thus suffices to prove the statement of the lemma for the iid array $X_{nk}'$ instead as claimed.

We henceforth assume that $\norm{X_{n1}} < \varepsilon$ a.s., where $\varepsilon > 0$ is sufficiently small so that $x \in U$ whenever $\norm{x} < \varepsilon$.
We first show~\eqref{eq_YnUnifBound}. Let $i \in \{1,\ldots, m\}$ such that $q_i \leq 1$. Then there exists $C_2 > 0$ such that
\[
\EEE{\norms{\Y^{n,i}_{k/n}}^{q_i}}  = \EEE{\norms{\sum_{a=1}^k X^i_{na}}^{q_i}} \leq \EEE{\sum_{a=1}^k |X^i_{na}|^{q_i}} \leq C_2k/n.
\]
where the second inequality is due to~\ref{point_thetaScale}. It follows by Markov's inequality that there exists $K > 0$ such that \eqref{eq_YnUnifBound} holds for all $n \geq 1$, $k \in \{1,\ldots, n\}$, and $\delta \in (0,1]$.

We now show~\eqref{eq_XnUnifBound}. By Lemma~\ref{lem:CBH}, it suffices to show that for all $i \in \{1,\ldots m\}$, $r \in \{1,\ldots, \floor{d_i}\}$,  multi-indexes $\alpha_1,\ldots, \alpha_r$ such that $\deg(\alpha_1) + \ldots +\deg(\alpha_r) = d_i$ (with $\alpha_1^i = 1$ in the case that $r=1$), there exists $K > 0$ such that for all $n \geq 1$, $k \in\{1,\ldots,n\}$ and $\delta \in (0,1]$
\begin{equation}\label{eq_expandedWords}
\PPP{\norms{\sum_{1 \leq a_1 < \ldots < a_r \leq k}X^{\alpha_1}_{na_1}\ldots X^{\alpha_r}_{na_r}}^{1/d_i} > \delta} \leq K\frac{k/n}{\delta^\gamma}.
\end{equation}

To this end, let us fix $i \in \{1,\ldots, m\}$, $r \in \{1,\ldots, \floor{d_i}\}$, and multi-indexes $\alpha_1,\ldots, \alpha_r$ such that $\deg(\alpha_1)+\ldots +\deg(\alpha_r) = d_i$. Consider first the case $r \geq 2$. Define
\[
\gamma_i := d_i \vee \max\left\{q_j d_j \midset j \in \{1,\ldots, m\}, d_j \leq d_i \right\}.
\]
By Markov's and Jensen's inequalities (observing that $\gamma_i \leq 2d_i$)
\begin{multline}\label{eq_boundrBig}
\PPP{\norms{\sum_{1 \leq a_1 < \ldots < a_r \leq k}X^{\alpha_1}_{na_1}\ldots X^{\alpha_r}_{na_r}}^{1/d_i} > \delta} \\
\leq \delta^{-\gamma_j} \EEE{\left(\sum_{1 \leq a_1 < \ldots < a_r \leq k} X^{\alpha_1}_{na_1}\ldots X^{\alpha_r}_{na_r}\right)^2}^{\gamma_i/2d_i}.
\end{multline}

To bound the last expression, for a multi-index $\alpha = (\alpha^1,\ldots,\alpha^m)$, denote $|\alpha| = \alpha^1 + \ldots + \alpha^m$. Note that due to the assumption $\norm{X_{n1}} < \varepsilon$ a.s.,~\ref{point_nEXiBound} is equivalent to
\begin{equation}\label{eq_nEWBound}
\sup_{|\alpha| = 1} \sup_{n \geq 1} n\norms{\EEE{X_{n1}^\alpha}} < \infty.
\end{equation}
Furthermore, by~\ref{point_thetaScale} and the Cauchy--Schwartz inequality,
\begin{equation}\label{eq_nAlphaBound}
\sup_{|\alpha| \geq 2}\sup_{n \geq 1}n\EEE{\norms{X^{\alpha}_{n1}}} < \infty.
\end{equation}

Consider now the expression
\begin{equation}\label{eq_squaresMons}
\EEE{\left(\sum_{1 \leq a_1 < \ldots < a_r \leq k} X^{\alpha_1}_{na_1}\ldots X^{\alpha_r}_{na_r}\right)^2}.
\end{equation}
Since $X_{n1},\ldots, X_{nn}$ are independent,~\eqref{eq_squaresMons} splits into a sum of terms of the form $\EEE{X_{n1}^{\beta_1}}\ldots \EEE{X_{nk}^{\beta_k}}$ with $\beta_i \geq 0$. Call the \emph{simple degree} of such a term the number of $\beta_i > 0$. The minimum simple degree of any term is evidently $r$ and the maximum is $2r$, and one readily sees that there exists $C_3 > 0$ such that for all $n \geq 1$ and $k \in \{ 1,\ldots, n\}$, the number of terms of simple degree $s \in \{r,\ldots, 2r\}$ is bounded above by $C_3 k^s$.
Furthermore, since $X_{n1},\ldots, X_{nn}$ are identically distributed, it follows from~\eqref{eq_nEWBound} and~\eqref{eq_nAlphaBound} that there exists $C_4 > 0$ such that the absolute value of every term of simple degree $s$ is bounded above by $C_4 n^{-s}$. Since $2 \leq r \leq s$ and $k \leq n$, it follows that
\[
\EEE{\left(\sum_{1 \leq a_1 < \ldots < a_r \leq k} X^{\alpha_1}_{na_1}\ldots X^{\alpha_r}_{na_r}\right)^2} \leq C_5 (k/n)^{2}.
\]
Therefore, from~\eqref{eq_boundrBig} and the fact that $d_i \leq \gamma_i \leq \gamma$, we obtain~\eqref{eq_expandedWords}. This completes the case $r \geq 2$.

It remains to consider the case $r = 1$. Define now $\gamma_i := d_i(q_i\vee 1)$. It holds that
\[
\PPP{\norms{\sum_{1 \leq a \leq k}X^{i}_{na}}^{1/d_i} > \delta} \leq \delta^{-\gamma_j} \EEE{\norms{\sum_{1 \leq a \leq k}X^i_{na}}^{q_i\vee 1}}.
\]
Denote $\mu_{n} = \EEE{X^i_{n1}}$. Then there exist $C_6, C_7 > 0$ such that
\begin{align*}
\EEE{\norms{\sum_{1 \leq a \leq k}X^i_{na}}^{q_i\vee 1}}
&= \EEE{\norms{\sum_{1 \leq a \leq k}X^i_{na} - \mu_n + \mu_n}^{q_i\vee 1}} \\
&\leq C_6\EEE{\norms{\sum_{1 \leq a \leq k}X^i_{na} - \mu_n}^{q_i\vee 1} + \norms{k/n}^{q_i\vee 1}} \\
&\leq C_7\left(\EEE{\sum_{1 \leq a \leq k}\norms{X^i_{na} - \mu_n}^{q_i\vee 1}} + (k/n)^{q_i\vee 1} \right),
\end{align*}
where the first inequality is due to~\eqref{eq_nEWBound}, and the second inequality is due to the (discrete) Burkholder--Davis--Gundy inequality and the fact that $q_i \leq 2$.
It now follows from~\ref{point_thetaScale} and~\eqref{eq_nEWBound} that
\begin{align*}
\EEE{\norms{\sum_{1 \leq a \leq k} X^i_{na}}^{q_i\vee 1}}
&\leq C_8\left(k\EEE{|X^i_{n1}|^{q_i\vee 1}} + k|\mu_n|^{q_i\vee 1} + (k/n)^{q_i\vee 1} \right) \\
&\leq C_{9}\left(k/n + k n^{-(q_i\vee 1)} + (k/n)^{q_i\vee 1} \right) \\
&\leq C_{10}\left(k/n\right).
\end{align*}
Since $\gamma_j \leq \gamma$, this completes the case $r=1$ and the proof of the lemma.
\end{proof}

As mentioned in Remark~\ref{remark:pBiggerThanN}, Corollary~\ref{cor:pVarTightCrit} and the bound~\eqref{eq_XnUnifBound} are now sufficient to prove Theorem~\ref{thm_cadlagPathTight} for the case that $p > d_m$.

\subsection{Proof of Theorem~\ref{thm_cadlagPathTight} in the case $p \leq d_m$}\label{subsec:pLess}

\begin{lemma}\label{lem_YProjTight}
Use the notation from Lemma~\ref{lem:boundsForTight} and suppose that \ref{point_hTight}, \ref{point_nEXiBound} and \ref{point_thetaScale} hold. For all $i \in \{1,\ldots, m\}$, it holds that $(\Y^{n,i}_h)_{n \geq 1, h \in [0,1]}$ is a tight collection of real random variables.
\end{lemma}

\begin{proof}
By Remark~\ref{remark:pBiggerThanN}, $(\X^n)_{n \geq 1}$ is a tight collection of $D_o([0,1],G)$-valued random variables, from which it follows that $(\max_{1 \leq k \leq n}|X^i_{nk}|)_{n \geq 1}$ is tight for all $i \in \{1,\ldots, m\}$.
We may thus suppose that $\norm{X_{n1}} \leq R$ a.s. for some large $R>0$ and all $n \geq 1$.

Consider the decomposition $X^i_{nk} = A_{nk} + B_{nk}$ where
\[
A_{nk} = X^i_{nk}\1{\norm{X_{nk}} < \varepsilon}
\]
and
\[
B_{nk} = X^i_{nk}\1{\varepsilon \leq \norm{X_{nk}} \leq R}.
\]
We take here $\varepsilon > 0$ sufficiently small so that $\norm{x} < \varepsilon$ implies $x \in U$.
It suffices to prove that $(\sum_{a=1}^k B_{na})_{n \geq 1, k \in \{1, \ldots, n\}}$ and $(\sum_{a=1}^k A_{na})_{n \geq 1, k \in \{1, \ldots, n\}}$ are tight collections of real random variables.

Let $C_1 = C_1(\varepsilon) > 0$ be such that $C_1\theta(x) > |x^i|\1{\varepsilon \leq \norm{x} \leq R}$ for all $x\in G$.
Since $\theta$ scales $X_{nk}$, it holds that
\[
\sup_{n \geq 1, k \in \{1,\ldots, n\}}\EEE{\norms{\sum_{a=1}^k B_{na}}} \leq \sup_{n \geq 1} C_1n\EEE{\theta(X_{n1})} < \infty,
\]
and thus $(\sum_{a=1}^k B_{na})_{n \geq 1, k \in \{1, \ldots, n\}}$ is tight.

Now observe that~\ref{point_nEXiBound} and~\ref{point_thetaScale} imply that $\sup_{n \geq 1} n\norms{\EEE{A_{n1}}} < \infty$. Moreover~\ref{point_thetaScale} implies that there exists $C_2 > 0$ such that $|\xi_i(x)|^2 \leq C_2\theta(x)$ for all $x \in G$ and $i \in \{1,\ldots, m\}$. Since $A_{na} = \1{\norm{X_{na}} < \varepsilon}\xi_i(X_{na})$, and $A_{n1},\ldots, A_{nn}$ are iid, it follows that
\[
\sup_{n \geq 1, k \in \{1,\ldots, n\}} \EEE{\left(\sum_{a=1}^k A_{na}\right)^2} \leq \sup_{n \geq 1} \sum_{a,b=1}^n \norms{\EEE{A_{na}A_{nb}}} < \infty,
\]
and thus $(\sum_{a=1}^k A_{na})_{n \geq 1, k \in \{1, \ldots, n\}}$ is also tight.
\end{proof}

For $i \in \{1,\ldots, m\}$, let $\g^{> i}$ be the subspace of $\g$ spanned by $\{u_j \mid j > i\}$. Note that $\g^{> i}$ is an ideal of $\g$, and so we can define the Lie algebra $\g^i = \g/\g^{> i}$ and the projection map $\pi^i : \g \mapsto \g^i$. The dilations $\delta_\lambda$ on $\g$ give rise to a natural family of dilations on $\g^{i}$, and thus to a homogeneous group $G^i$ associated with $\g^i$. Equivalently, $G^i = \g/\g^{>i}$, where we have identified $\g$ with $G$ and $\g^{> i}$ with a normal subgroup of $G$.  We implicitly equip $G^i$ with an arbitrary sub-additive homogeneous norm $\normc$. For notational convenience, we also let $G^0 = \{1\}$ be the trivial group and $\pi^0 : G \mapsto G^0$ the trivial map.

\begin{corollary}\label{cor_indvComponentsTight}
Use the notation from Lemma~\ref{lem:boundsForTight} and suppose that~\ref{point_hTight}, \ref{point_nEXiBound} and \ref{point_thetaScale} hold.

\begin{enumerate}[label={(\roman*)}]
\item \label{point_ZnTight} Let $i \in \{1,\ldots, m\}$ and $p > d_i \vee \kappa$. Then $(\norm{\pi^{i}\X^n}_{\pvar;[0,1]})_{n \geq 1}$ is tight.

\item \label{point_YsTight} For every $i \in \{1, \ldots, m\}$ such that $q_i \leq 1$, $(\norm{\Y^{n,i}}_{\var{p};[0,1]})_{n \geq 1}$ is tight for all $p > q_i$.
\end{enumerate} 
\end{corollary}

\begin{proof}
\ref{point_ZnTight} Observe that $\pi^i\X^n$ is the random walk associated with the $G^{i}$-valued iid array $\pi^{i} X_{nk}$, from which the conclusion follows by Corollary~\ref{cor:pVarTightCrit} and the bound~\eqref{eq_XnUnifBound} of Lemma~\ref{lem:boundsForTight} (cf. Remark~\ref{remark:pBiggerThanN}).

\ref{point_YsTight} From Corollary~\ref{cor:pVarTightCrit} and the bound~\eqref{eq_YnUnifBound} of Lemma~\ref{lem:boundsForTight}, it suffices to check that condition~\ref{point:coordsTight} of Corollary~\ref{cor:pVarTightCrit} holds for the processes $(\Y^{n,i})_{n \geq 1}$. However this follows from Lemma~\ref{lem_YProjTight}.
\end{proof}

Recall that we identify $G$ with $\g$ via the $\exp$ map. For functions $\z : [0,T] \mapsto G$ and $\y : [0,T] \mapsto \R$, define the function
\[
\x = \z + \y : [0,T] \mapsto G, \; \; \x_t = \z_t + \y_t u_m.
\]
(where addition is taken in $\g$). The following lemma is a simple consequence of the fact that $(\x_s^{-1}\x_t)^i = (\z_s^{-1}\z_t)^i$ for all $i \in \{1,\ldots, m-1\}$, and
\[
(\x_s^{-1}\x_t)^m = (\z_s^{-1}\z_t)^m + \y_t - \y_s.
\]

\begin{lemma}\label{lem_centralPathBound}
Let $\z : [0,T] \mapsto G$ and $\y : [0,T] \mapsto \R$ be functions, and let $\x = \z + \y$. Then for any $p > 0$ there exists $C=(p,G) > 0$ such that
\[
\norm{\x}_{\pvar;[0,T]} \leq C\left(\norm{\z}_{\pvar;[0,T]} + \norm{\y}^{1/d_m}_{\var{p/d_m};[0,T]} \right).
\]
\end{lemma}

\begin{lemma}\label{lem_ZYVarControlXVar}
Let $p > 0$ and $i \in \{1,\ldots, m\}$ be the largest index such that $d_i \leq p$ (with $i=0$ if no such index exists). Consider elements $x_1,\ldots, x_n \in G$ and let $\x \in D_o([0,1], G)$ be the associated walk.
For $j \in \{i + 1, \ldots, m\}$, let $\y^j \in D_o([0,1], \R)$ be the walk associated with the real numbers $x^j_1,\ldots, x^j_n$.

Then there exists $C = C(p,G) > 0$, such that
\[
\norm{\x}_{\pvar;[0,1]} \leq C\left(\norm{\pi^{i}\x}_{\pvar;[0,1]} + \sum_{j = i + 1}^m \norm{\y^j}^{1/d_j}_{\var{p/d_j};[0,1]}\right).
\]
\end{lemma}

\begin{proof}
By induction on $m$ and Lemma~\ref{lem_centralPathBound}, it suffices to show that if $p < d_m$ and $x^m_k = 0$ for all $k \in \{1,\ldots, n\}$, then $\norm{\x}_{\pvar;[0,1]} \leq C_1\norm{\pi^{m-1} \x}_{\pvar;[0,1]}$. This in turn follows from the CBH formula~\eqref{eq:CBH} and an application of Young's partition coarsening argument (see, e.g.,~\cite[p. 50]{Lyons07}).
\end{proof}

We now have all the ingredients for the proof of Theorem~\ref{thm_cadlagPathTight}.

\begin{proof}[Proof of Theorem~\ref{thm_cadlagPathTight}]
The fact that $(\X^n)_{n \geq 1}$ is a tight collection of $D_o([0,1],G)$-valued random variables follows directly from Corollary~\ref{cor:pVarTightCrit} and the bound~\eqref{eq_XnUnifBound} of Lemma~\ref{lem:boundsForTight} (cf. Remark~\ref{remark:pBiggerThanN}).

Let $p > \kappa$. Decreasing $p$ if necessary, we may suppose $p \neq d_i$ for all $i \in \{1,\ldots, m\}$. Let $i \in \{1,\ldots, m\}$ be the largest index such that $d_i < p$ (with $i=0$ if no such index exists). Define $\Y^{n,j}$ as in Lemma~\ref{lem:boundsForTight}, and note that $q_j < p/d_j < 1$ for all $j \in \{i + 1,\ldots, m\}$. It follows by Corollary~\ref{cor_indvComponentsTight} that $(\norm{\pi^i \X^n}_{\pvar;[0,1]})_{n \geq 1}$ and $(\norm{\Y^{n,j}}_{\var{p/d_j};[0,1]})_{n \geq 1}$ are tight for all $j \in \{i+1,\ldots, m\}$.
We conclude by Lemma~\ref{lem_ZYVarControlXVar} that $(\norm{\X^n}_{\pvar;[0,1]})_{n \geq 1}$ is also tight.
\end{proof}

\section{L{\'e}vy processes in homogeneous groups}\label{sec:LevyProcesses}

\subsection{Finite $p$-variation of L{\'e}vy processes} \label{subsec_finitePVar}

Consider a homogeneous group $G$ and recall the notation of Section~\ref{sec:Hom}. Recall also the definitions of $\Gamma_i$, $J$, and $K$ from Section~\ref{subsec:approxWalk}. The following is the main result of this subsection.

\begin{theorem}\label{thm_LevyFinitepVar}
Let $p > 0$ and $\X$ be a L{\'e}vy process in $G$ with triplet $(A, B, \Pi)$.
\begin{enumerate}[label={(\arabic*)}]
\item \label{point_finitepVar} Then $\norm{\X}_{\pvar;[0,1]} < \infty$ a.s. provided that all of the following hold:
\begin{enumerate}[label={(\roman*)}]
\item \label{point_pBigger2} $p > 2d_j$ for all $j \in J$;

\item \label{point_pBiggerUk} $p > d_k$ for all $k \in K$;

\item \label{point_pBiggerBeta} $p/d_i > \sup \{\Gamma_i\}$ for all $i \in \{1, \ldots, m\}$.
\end{enumerate}

\item \label{point_infinitepVar} Then $\norm{\X}_{\pvar;[0,1]} = \infty$ a.s. provided that one of the following holds:
\begin{enumerate}[resume,label={(\roman*)}]
\item $p \leq 2d_j$ for some $j \in J$;

\item \label{point_pLessUk} $p < d_k$ for some $k \in K$;

\item \label{point_pInGamma} $p/d_i \in \Gamma_i$ for some $i \in \{1,\ldots, m\}$.
\end{enumerate}
\end{enumerate}
\end{theorem}

\begin{remark}
Note that
Theorem~\ref{thm_LevyFinitepVar} does not completely determine all values of $p$ for which $\norm{\X}_{\pvar;[0,1]} < \infty$ a.s. (e.g., when $p/d_i = \sup \{\Gamma_i\} \notin \Gamma_i$ for some $i \in \{1, \ldots, m\}$).
Comparing Theorem~\ref{thm_LevyFinitepVar} with known results for $\R$-valued L{\'e}vy processes~\cite{Bretagnolle72II}, we suspect that~\ref{point_pBiggerUk} and~\ref{point_pBiggerBeta} can be replaced by $p \geq d_k, \forall k \in K$, and $p/d_i \notin \Gamma_i, \forall i \in \{1,\ldots, m\}$, respectively, which would complete the characterisation.
\end{remark}

\begin{remark}
In~\cite{FrizShekhar12}, the authors determined sufficient conditions under which a L{\'e}vy process in the step-$2$ free nilpotent Lie group $G^2(\R^d)$ possesses finite $p$-variation for $p \in (2,3)$, along with a partial converse that their conditions cannot in general be weakened (\cite[Theorem~50]{FrizShekhar12}). In this context, Theorem~\ref{thm_LevyFinitepVar} generalises this result to all $N \geq 1$ and $p > 0$ and provides a sharp converse. In particular, the Carnot--Carath{\'e}odory Blumenthal--Getoor index $\beta$ introduced in~\cite{FrizShekhar12} for a L{\'e}vy measure on $G^N(\R^d)$ relates to our definition of $\Gamma_i$ by $\beta = \max\{d_1\sup\{\Gamma_1\},\ldots, d_m\sup\{\Gamma_m\}\}$, in which case~\ref{point_pBiggerBeta} reads $p > \beta$.
\end{remark}

For the proof of Theorem~\ref{thm_LevyFinitepVar}, we require the following lemma.

\begin{lemma}\label{lem_associatedWalkTight}
Let $\X$ be a L{\'e}vy process in $G$ with triplet $(A,B,\Pi)$. Assume $p > 0$ satisfies~\ref{point_pBigger2},~\ref{point_pBiggerUk}, and~\ref{point_pBiggerBeta} of Theorem~\ref{thm_LevyFinitepVar}.

Let $X_{nj}$ be the associated iid array constructed in Section~\ref{subsec:approxWalk} and $\X^n$ the associated random walk. Then $(\norm{\X^n}_{\pvar;[0,1]})_{n \geq 1}$ is tight.
\end{lemma}

\begin{proof}
Let $0 < p' < p$ such that $p'$ also satisfies~\ref{point_pBigger2},~\ref{point_pBiggerUk}, and~\ref{point_pBiggerBeta} of Theorem~\ref{thm_LevyFinitepVar}. For all $i \in \{1,\ldots, m\}$, define $q_i := 2\wedge(p'/d_i)$, and let $\theta$ be a scaling function on $G$ such that $\theta \equiv \sum_{i = 1}^m |\xi_i|^{q_i}$ in a neighbourhood of $1_G$.

Observe that $q_i \notin \Gamma_i$ for all $i \in \{1,\ldots, m\}$, $q_j = 2$ for all $j \in J$, and $q_k > 1$ for all $k \in K$. Thus, by Lemma~\ref{lem_thetaScales}, $\theta$ scales the array $X_{nj}$.
Moreover, since $\X^n \convd \X$ as $D_o([0,1],G)$-valued random variables, it follows that the array $X_{nj}$ satisfies the conditions of Theorem~\ref{thm_cadlagPathTight} with the above $\theta$ and $q_1,\ldots, q_m$ (see Remark~\ref{remark:convLevyImpliesConds}). Since $p > \max\{q_1 d_1,\ldots, q_m d_m\}$, it follows that $(\norm{\X^n}_{\pvar;[0,1]})_{n \geq 1}$ is tight.
\end{proof}

\begin{proof}[Proof of Theorem~\ref{thm_LevyFinitepVar}]
\ref{point_finitepVar} follows from Lemma~\ref{lem_associatedWalkTight} and part~\ref{point:finitepVar} of Proposition~\ref{prop:finiteVarLim}, while~\ref{point_infinitepVar} follows directly from Corollary~\ref{cor_xiInfVar} and Proposition~\ref{prop_BlumenthalGetoor}.
\end{proof}

\subsection{Convergence in $p$-variation}\label{subsec:convFixedPathFunc}

In this subsection we consider continuous random paths $(\X^{n,\phi})_{n \geq 1}, \X^\phi$, constructed from a random walk $\X^n$ and a L{\'e}vy process $\X$ by connecting their left- and right-limits with a path function $\phi$, and give conditions under which $\X^{n,\phi} \convd \X^\phi$ as $C^{\pvar}([0,1],G)$-valued random variables. All relevant material on path functions is collected in Appendix~\ref{appendix:PathFuncs}.

\begin{theorem}\label{thm:convToLevy}
Let $X_{nj}$ be an iid array in $G$ and $\X^n$ the associated random walk. Let $\X$ be a L{\'e}vy process in $G$ with triplet $(A,B,\Pi)$. Suppose that $\X^n \convd \X$ as $D_o([0,1],G)$-valued random variables and that $\theta$ scales $X_{nj}$, where $\theta \equiv \sum_{i=1}^m |\xi_i|^{q_i}$ in a neighbourhood of $1_G$ for some $0 < q_i \leq 2$.

Let $W \subseteq G$ be a closed subset such that $\supp(\Pi) \subseteq W$ and $X_{n1} \in W$ a.s. for all $n \geq 1$. Let $p > \max\{1, q_1 d_1,\ldots,q_m d_m\}$ and $\phi : W \mapsto C_o^{\pvar}([0,1],G)$ a $p$-approximating, endpoint continuous path function.

Then $\norm{\X^\phi}_{\pvar;[0,1]} < \infty$ a.s., and
for every $p' > p$, $\X^{n,\phi} \convd \X^\phi$ as $C_o^{0,\pprimevar}([0,1], G)$-valued random variables.
\end{theorem}

\begin{remark}\label{remark:jumpsInSupp}
In the statement of Theorem~\ref{thm:convToLevy}, note that, a.s., $\X_{t-}^{-1}\X_t \in \supp(\Pi)$ for every jump time $t$ of $\X$ (e.g., \cite[Proposition~1.4]{Liao04}). Hence, for any (measurable) path function $\phi$ defined on $\supp(\Pi)$, $\X^\phi$ is indeed a well-defined $C_o([0,1],G)$-valued random variable.
\end{remark}

\begin{proof}
By Theorem~\ref{thm_cadlagPathTight}, it holds that $(\norm{\X^n}_{\pvar;[0,1]})_{n \geq 1}$ is tight, and thus, by Proposition~\ref{prop_pvarPsi}, $(\norm{\X^{n,\phi}}_{\pvar;[0,1]})_{n \geq 1}$ is also tight.
Since $\phi$ is endpoint continuous on $W$, it follows by Proposition~\ref{prop_convFromDToC} that $\X^{n,\phi} \convd \X^{\phi}$ as $C_o([0,1],G)$-valued random variables. The conclusion now follows from Proposition~\ref{prop:finiteVarLim}.
\end{proof}

\subsection{Applications to rough paths theory}\label{subsec:RPs}

We apply the results so far developed in the paper to the theory of rough paths and stochastic flows. Following Example~\ref{ex:gradedLie}, denote by $G^N(\R^d)$ the step-$N$ free nilpotent Lie group over $\R^d$ and let $\g^N(\R^d)$ be its Lie algebra. For the remainder of the paper, unless otherwise stated, we shall always let $G = G^N(\R^d)$ and $\g = \g^N(\R^d)$. Being a Carnot group, $G$ comes equipped with a natural homogeneous structure and we note that $u_1,\ldots, u_d$ can be identified with a basis of $\R^d$.

For $1 \leq p < N+1$, we let $WG\Omega_p(\R^d) := C^{\pvar}([0,T], G)$, equipped with the metric $d_{\pvar;[0,T]}$, denote the space of weakly geometric $p$-rough paths. Given an element $\x \in WG\Omega_p(\R^d)$, and a collection $(f_i)_{i=1}^d$ of vector fields in $\Lip^{\gamma+k-1}(\R^e)$ for $\gamma > p \geq 1$ and an integer $k \geq 1$, there is a unique solution to the rough differential equation (RDE)
\begin{equation}\label{eq:RDE}
d\y_t = f(\y_t)d\x_t, \; \; \y_0 \in \R^e.
\end{equation}
We refer to~\cite{FrizVictoir10} for further details on (geometric) rough paths theory.

\subsubsection{Stochastic flows}

\label{subsubsec:StochFlows}

Let $U^\x_{T\leftarrow 0} : \y_0 \mapsto \y_T$ denote the flow associated to~\eqref{eq:RDE}, which we recall is an element of $\Diff^k(\R^e)$, the group of $C^k$-diffeomorphisms of $\R^e$. Recall that the map $U^\cdot_{T\leftarrow 0} : WG\Omega_p(\R^d) \mapsto \Diff^k(\R^e)$ is a continuous function on $WG\Omega_p(\R^d)$ when $\Diff^k(\R^e)$ is equipped with the $C^k$-topology (\cite[Theorem~11.12]{FrizVictoir10}). The following result is now an immediate corollary of Theorem~\ref{thm:convToLevy}.

\begin{corollary}\label{cor:convFlowsToLevy}
Suppose the assumptions of Theorem~\ref{thm:convToLevy} are verified for some $1 \leq p < N+1$. Let $\gamma > p$, $k \geq 1$ an integer, and $(f_i)_{i=1}^d$ a collection of vector fields in $\Lip^{\gamma+k-1}(\R^e)$.
Let $U^\cdot_{1\leftarrow 0} : WG\Omega_p(\R^d) \mapsto \Diff^k(\R^e)$ be the associated flow map.

Then $U^{\X^{n,\phi}}_{1\leftarrow 0} \convd U^{\X^\phi}_{1\leftarrow 0}$ as $\Diff^k(\R^e)$-valued random variables.
\end{corollary}

We demonstrate how one can apply Corollary~\ref{cor:convFlowsToLevy} to show weak convergence of stochastic flows in the following three examples, the first of which extends a result of Kunita~\cite{Kunita95}.

\begin{example}[Linear interpolation, Kunita~\cite{Kunita95}]\label{ex_piecewiseLinear}
Let $Y_{n1},\ldots, Y_{nn}$ be an iid array in $\R^d$ such that the associated random walk $\Y^n$ converges in law as a $D_o([0,1],\R^d)$-valued random variable to a L{\'e}vy process $\Y$ in $\R^d$.

We claim that ODE flows driven by the piecewise linear interpolation of the random walk $\Y^n$ along $\Lip^{\gamma+k-1}$ vector fields, for any $\gamma > 2$, $k \geq 1$, converge in law as $\Diff^k(\R^e)$-valued random variables.

Indeed, setting $G := G^2(\R^d)$, consider the $G$-valued iid array $X_{nj} := e^{Y_{nj}}$.
It follows that $X_{nj}$ is scaled by any scaling function $\theta$ on $G$ for which $\theta \geq \sum_{i=1}^d |\xi_i|^2$.
Moreover, using the fact that $\xi_i\circ\exp \in C^\infty_c(\R^d)$, one can readily see by Theorem~\ref{thm_iidConvLevy} that $\X^n\convd \X$ as $D_o([0,1],G)$-valued random variables, where $\X$ is a $G$-valued L{\'e}vy process.
Finally, consider the $1$-approximating, endpoint continuous path function
\[
\phi : \exp(\R^d) \mapsto C_o([0,1], G), \; \; \phi(e^x)_t = e^{tx}.
\]
Then $\X^{n,\phi}$ is (a reparametrisation of) the lift of the piecewise linear interpolation of $\Y^n$. Furthermore, the conditions of Theorem~\ref{thm:convToLevy} are satisfied for all $p > 2$, so that $\X^{n,\phi} \convd \X^{\phi}$ as $C_o^{0,\pvar}([0,1],G)$-valued random variables, from which the desired claim follows (Corollary~\ref{cor:convFlowsToLevy}).
\end{example}

\begin{remark}\label{remark:MarcusSDE}
In the previous example, it is easy to see that RDEs driven by $\X^\phi$ coincide (up to reparametrisation) with general (Marcus) RDEs driven by $\X$ (in the sense of~\cite[Section~6]{FrizShekhar12}) and thus with Marcus SDEs driven by $\Y$.
\end{remark}

\begin{remark}
The previous example extends the main result of Kunita~\cite{Kunita95} (Theorem~4 and its Corollary). The main restriction of Kunita's result is the assumption that the vector fields $f_1,\ldots, f_d$, along which $\Y^n$ drives an ODE, generate a finite dimensional Lie algebra,
which essentially allows one to reduce the problem to a random walk on a Lie group (see~\cite[p. 340]{Kunita95}). Our approach, based on convergence under rough path topologies, bypasses this restriction and provides a natural interpretation of the limiting stochastic flow as the solution of an RDE.
\end{remark}

\begin{remark}\label{remark:BFH}
Breuillard, Friz and Huesmann~\cite{Breuillard09} showed a result analogous to the above example in a special case where the limiting L{\'e}vy process $\Y$ is Brownian motion.
The main analytic tool used in~\cite{Breuillard09} is the Kolmogorov--Lamperti criterion to show tightness of $(\norm{\Y^n}_{\pHol;[0,1]})_{n \geq 1}$. This is of course stronger than tightness of $(\norm{\Y^n}_{\pvar;[0,1]})_{n \geq 1}$, and cannot hold whenever the limiting L{\'e}vy process has jumps, which demonstrates an example where the tightness criterion Theorem~\ref{thm:pVarTightCrit} can be used as an effective alternative to the classical Kolmogorov--Lamperti criterion.
\end{remark}

In the following example we demonstrate how Example~\ref{ex_piecewiseLinear} generalises to non-linear interpolations with essentially no extra effort.

\begin{example}[Non-linear interpolation]\label{ex_nonLinearInter}
As in Example~\ref{ex_piecewiseLinear}, let $Y_{nj}$ be an iid array in $\R^d$ such that $\Y^n \convd \Y$ for a L{\'e}vy process $\Y$ in $\R^d$.

Instead of piecewise interpolations, consider now any $q$-approximating endpoint continuous path function $\psi : \R^d \mapsto C_o^{\qvar}([0,1], \R^d)$ for some $1 \leq q < 2$.
Set again $G = G^2(\R^d)$ and define the injective map $f : \R^d \mapsto G$ by
\[
f(x) = \Lyons_2(\psi(x))_1,
\]
where $S_2 : C_o^{\qvar}([0,1],\R^d) \mapsto C_o^{\qvar}([0,1],G)$ denotes the level-$2$ lifting map.

Consider the iid array $X_{nj} := f(Y_{nj})$.
It follows readily from the assumption that $\psi$ is $q$-approximating that $X_{nj}$ is again scaled by any scaling function $\theta$ on $G$ for which $\theta \geq \sum_{i=1}^d |\xi_i|^2$.
We now make the assumption on $\psi$ and $Y_{n1}$ that for all $i,j \in \{1,\ldots m\}$ the following limits exist:
\begin{align*}
D^i &:= \lim_{n \rightarrow \infty} n\EEE{\xi_i(f(Y_{n1}))}, \\
C^{i,j} &:= \lim_{n \rightarrow \infty} n\EEE{\xi_i(f(Y_{n1}))\xi_j(f(Y_{n1}))} - \int_{\R^d} \xi_i(f(x))\xi_j(f(x)) \Pi(dx).
\end{align*}
This occurs, for example, whenever every $\xi_i \circ f$ is twice differentiable at zero, but in general will depend on the array $Y_{nj}$ and the path function $\psi$.

Under this assumption, it follows from Theorem~\ref{thm_iidConvLevy} that the random walk $\X^n$ associated with the array $X_{nj}$ converges in law to the L{\'e}vy process $\X$ with triplet $(C,D,\Xi)$, where $\Xi$ is the pushforward of $\Pi$ by $f$.

Define now the $q$-approximating, endpoint continuous path function $\phi : f(\R^d) \mapsto C_o^{\qvar}([0,1],G)$ by
\[
\phi(f(x)) = \Lyons_2(\psi(x)).
\]
Observe that the conditions of Theorem~\ref{thm:convToLevy} are again satisfied for all $p > 2$, so that $\X^{n,\phi} \convd \X^{\phi}$ as $C_o^{0,\pvar}([0,1],G)$-valued random variables.

Note that $\X^{n,\phi}$ is, up to reparametrisation, the lift of $\Y^{n,\psi}$ (which is itself, up to reparametrisation, the random walk $\Y^n$ interpolated by the path function $\psi$).
It follows that ODE flows driven by $\Y^{n,\psi}$ along $\Lip^{\gamma+k-1}$ vector fields, for any $\gamma > 2$, $k \geq 1$, converge in law as $\Diff^k$-valued r.v.'s to the corresponding RDE flow driven by $\X^\phi$ (Corollary~\ref{cor:convFlowsToLevy}).
\end{example}

\begin{remark}\label{remark:McShane}
McShane~\cite{McShane72} considered non-linear interpolations of the increments of Brownian motion and showed strong convergence of the corresponding ODEs to the associated Stratonovich SDE with an adjusted drift. We note that the family of path functions $\psi$ to which the above example applies includes the non-linear interpolations considered by McShane (\cite[p. 285]{McShane72}) (provided that $Y_{nk}$ are also sufficiently well behaved, e.g., the increments of Brownian motion, to ensure that the limits $C^{i,j}$ and $D^i$ exist). The above example can thus be seen as a weak convergence analogue for general L{\'e}vy processes of the results in~\cite{McShane72}.
In a similar way, the following example is analogous to the results of Sussman~\cite{Sussmann91} on non-linear approximations of Brownian motion.
\end{remark}

\begin{example}[Perturbed walk]\label{ex_pertubedWalk}
As in Examples~\ref{ex_piecewiseLinear} and~\ref{ex_nonLinearInter}, let $Y_{nj}$ be an iid array in $\R^d$ such that $\Y^n \convd \Y$ for a L{\'e}vy process $\Y$ in $\R^d$.

Let $N \geq 2$ and as before denote $G = G^N(\R^d)$ and $\g = \g^N(\R^d)$. Fix a path $\gamma \in C_o^{\onevar}([0,1], \R^d)$ such that $v := \log(\Lyons_N(\gamma)_{0,1})$ is in the center of $\g$ (that is, $v^i = 0$ for all $i \in \{1,\ldots, m\}$ such that $d_i < N$).

In this example we wish to consider the random path $\ZZ^n \in C^{\onevar}([0,1], \R^d)$ defined by linearly joining the points of $\Y^n$, and, between each linear chord, running along the path $n^{-1/N}\gamma$.

Define the closed subset
\[
W := \left\{\exp(y)\exp(\lambda v) \midset y \in \R^d, \lambda \geq 0\right\} \subseteq G.
\]
Note that for every $x \in W$ decomposes uniquely as $x = \exp(y)\exp(\lambda v)$ for some $y \in \R^d$ and $\lambda \geq 0$.
Define then the $1$-approximating, endpoint continuous path function $\phi : W \mapsto C_o^{\onevar}([0,1],G)$ by
\[
\phi(\exp(y)\exp(\lambda v))_t = 
\begin{cases} \exp(2t y) &\mbox{if } t \in [0,1/2] \\ 
\exp(y)\Lyons_N(\lambda^{1/N}\gamma)_{2t-1} &\mbox{if } t \in (1/2,1].
\end{cases}
\]
Consider the $G$-valued iid array $X_{nj} := \exp(Y_{nj})\exp(n^{-1} v)$ and the associated random walk $\X^n$. Observe that $\X^{n,\phi}$ is (a reparametrisation of) the level-$N$ lift of the path $\ZZ^n$ described above.

We now claim that $\X^n \convd \X$ for a L{\'e}vy process $\X$ in $G$. A straightforward way to show this is to take local coordinates $\sigma_1,\ldots, \sigma_d \in C^\infty_c(\R^d)$ so that $\sigma := \sum_{i=1}^d \sigma_iu_i$ is the identity in a neighbourhood of zero, and write the triplet of $\Y$ as $(A,B,\Pi)$ with respect to $\sigma_1,\ldots, \sigma_d$.
Define the functions
\[
f_n : \R^d \mapsto \g, \; \; f_n(y) = \xi(e^y e^{v/n}),
\]
so that $\xi(X_{nk}) = f_n(Y_{nk})$.
Note that, since $v$ is in the centre of $\g$, there exists a neighbourhood of zero $V \subset \R^d$ and $n_0 > 0$ such that for all $n \geq n_0$
\[
f_n(y) = \sum_{i=1}^d \sigma_i(y) u_i + \xi(e^{v/n}) + h_n(y),
\]
where $h_n \equiv 0$ on $V$. It readily follows that
\[
\lim_{k\rightarrow \infty} \sup_{n \geq 1} k\norms{\EEE{f_n(Y_{k1}) - \xi(e^{v/n})} - Q_n} = 0,
\]
where
\[
Q_n := \sum_{i=1}^d B^iu_i + \int_{\R^d} h_n(y)\Pi(dy) \in \g.
\]

Observe now that for all $y \in \R^d$, $\lim_{n\rightarrow \infty}h_n(y) = \xi(e^y) - \sigma(y)$, so that by dominated convergence,
\[
\lim_{n \rightarrow \infty} Q_n = \sum_{i=1}^d B^iu_i + \int_{\R^d} \left(\xi(e^y) - \sigma(y)\right)\Pi(dy) =: Q \in \g,
\]
from which it follows that
\[
\lim_{n \rightarrow \infty} n\EEE{f_n(Y_{n1}) - f_n(0)} = Q.
\]
Since $nf_n(0) = v$ for all $n$ sufficiently large, we obtain that the following limit exists:
\[
D := \lim_{n \rightarrow \infty}n\EEE{\xi(X_{n1})} \in \g.
\]
Furthermore, letting $\Xi$ denote the pushforward of $\Pi$ by $\exp$, one can show in exactly the same way that
\[
C^{i,j} := \lim_{n \rightarrow\infty} n\EEE{\xi_i(X_{n1})\xi_j(X_{n1})} - \int_{G}\xi_i(x)\xi_j(x)\Xi(dx) 
\]
exists for all $i,j \in \{1,\ldots, m\}$, and that
\[
\lim_{n \rightarrow \infty} n\EEE{f(X_{n1})} = \int_{G} f(x)\Xi(dx)
\]
for every $f \in C_b(G)$ which is identically zero on a neighbourhood of $1_G$.
It follows by Theorem~\ref{thm_iidConvLevy} that $\X^n \convd \X$ as claimed, where $\X$ is the L{\'e}vy process with triplet $(C,D,\Xi)$.

Finally, one readily sees that $X_{nj}$ is scaled by any scaling function $\theta$ on $G$ for which
\[
\theta \geq \sum_{i=1}^d|\xi_i|^2 + \sum_{\substack{1 \leq i \leq m \\ d_i = N}} |\xi_i|.
\]

It now follows by Theorem~\ref{thm:convToLevy} that for all $p > N$, $\X^{n,\phi} \convd \X^{\phi}$ as $C_o^{0,\pvar}([0,1],G)$-valued r.v.'s.
In particular,
ODE flows driven by the random paths $\ZZ^n$ along $\Lip^{\gamma+k-1}$ vector fields, for any $\gamma > N$, $k \geq 1$, converge in law as $\Diff^k$-valued r.v.'s to the corresponding RDE flow driven by $\X^\phi$ (Corollary~\ref{cor:convFlowsToLevy}).
\end{example}

\begin{remark}
Note that the previous Example~\ref{ex_piecewiseLinear} is a special case of Example~\ref{ex_pertubedWalk} by taking $v = 0$ and $\gamma$ the constant path $\gamma \equiv 0$. Building on Remark~\ref{remark:MarcusSDE}, one can verify that RDEs driven by $\X^\phi$ coincide (up to reparametrisation) with the associated Marcus SDEs driven by $\Y$ with an adjusted drift given by appropriate $N$-th level Lie brackets of the driving vector fields (cf.~\cite{FrizOberhauser09} and~\cite[Section~13.3.4]{FrizVictoir10}).
\end{remark}

\subsubsection{The L{\'e}vy--Khintchine formula for L{\'e}vy rough paths}\label{subsubsec:LKFormula}

In this subsection we determine a formula for the characteristic function (in the sense of~\cite{ChevyrevLyons16}) of the signature of a L{\'e}vy rough path.

Recall that for every $\x \in WG\Omega_p(\R^d)$, there exists an element
\[
S(\x)_{0,T} = (1,S(\x)^1_{0,T},S(\x)^2_{0,T}, \ldots) \in T((\R^d)) = \prod_{k=0}^\infty (\R^d)^{\otimes k},
\]
called the \emph{signature} of $\x$, where $S(\x)_{0,T}^k$ encodes all the $k$-fold iterated integrals of $\x$.
A fundamental result in rough paths theory is that $S(\x)_{0,T}$ belongs to a certain group $G(\R^d)$ contained in the set of group-like elements of $T((\R^d))$ (for the tensor Hopf algebra structure). Furthermore, for every linear map $f \in \LLL(\R^d,\LLL(\R^e, \R^e))$, the series $\sum_{k=0}^\infty f^{\otimes k}(S(\x)_{0,T}^k)$ converges absolutely to an operator $f(S(\x)_{0,T}) \in \LLL(\R^e,\R^e)$ (which is precisely the flow $U^\x_{T\leftarrow 0}$ associated with the RDE~\eqref{eq:RDE} upon treating $f$ as a collection of linear vector fields on $\R^e$).

For a finite-dimensional complex Hilbert space $H$, let $\uu(H) \subset \LLL(H,H)$ denote the Lie algebra of anti-Hermitian operators on $H$ and $\UU(H) \subset \LLL(H,H)$ the group of unitary operators on $H$. Note that every $f \in \LLL(\R^d, \uu(H))$ naturally induces a map $f : G(\R^d) \mapsto \LLL(H,H)$ (which is continuous for the topology on $G(\R^d)$ introduced in~\cite{ChevyrevLyons16}) given by $f(x) = \sum_{k=0}^\infty f^{\otimes k}(x^k)$, where $x^k \in (\R^d)^{\otimes k}$ denotes the level-$k$ projection of $x$. Note that $f$ satisfies $f(x) \in \UU(H)$ and $f(xy) = f(x)f(y)$ for all $x,y \in G(\R^d)$, i.e., $f : G(\R^d) \mapsto \UU(H)$ is a unitary representation of $G(\R^d)$.

One of the main results of~\cite{ChevyrevLyons16} is that for any $WG\Omega_p(\R^d)$-valued random variable $\X$, the following characteristic function
\begin{equation}\label{eq:charFunc}
\begin{split}
\LLL(\R^d,\uu(H)) &\mapsto \LLL(H,H),\\
f &\mapsto \EEE{f(S(\X)_{0,T})},
\end{split}
\end{equation}
where $H$ varies over all finite dimensional complex Hilbert spaces, uniquely determines $S(\X)_{0,T}$ as a $G(\R^d)$-valued random variable $G(\R^d)$ (more generally, this result holds for every $G(\R^d)$-valued random variable).

\begin{remark}
Boedihardjo et al.~\cite{Boedihardjo16} have recently established a conjecture of Hambly--Lyons~\cite{Hambly10} on the kernel of the map $S : WG\Omega_p(\R^d) \mapsto T((\R^d))$. A consequence of the main result of~\cite{Boedihardjo16} is that for all $\x,\y \in WG\Omega_p(\R^d)$, $S(\x)_{0,T} = S(\y)_{0,T} \Leftrightarrow U^\x_{T\leftarrow 0} = U^\y_{T\leftarrow 0}$ for all collections $(f_i)_{i=1}^d$ of vector fields in $\Lip^{\gamma}(\R^e)$ with $\gamma > p$ (not necessarily linear). In combination with the results from~\cite{ChevyrevLyons16}, it follows that for any $WG\Omega_p(\R^d)$-valued random variable $\X$, knowledge of the map~\eqref{eq:charFunc} uniquely determines the law of every RDE driven by $\X$.
\end{remark}

We now state the aforementioned formula for the characteristic function of the signature of a L{\'e}vy rough path. For a subset $W\subseteq G$, path function $\phi : W \mapsto C^{\pvar}_o([0,1], G)$, and a linear map $f \in \LLL(\R^d,\LLL(\R^e,\R^e))$, we adopt the shorthand notation
\[
f_{\phi} : W \mapsto \LLL(\R^e,\R^e), \; \; f_\phi(x) := f(S(\phi(x))_{0,1}).
\]
By interpolation (Lemma~\ref{lem:interpolate}), one can readily verify that $f_\phi$ is continuous whenever $\phi$ is $p$-approximating and endpoint continuous. Finally, we canonically treat $\g = \g^N(\R^d)$ as a subspace of the tensor algebra $T(\R^d)$, so that for any Lie algebra $\h$, every $f \in \LLL(\R^d, \h)$ extends uniquely to a linear map $f : \g \mapsto \h$.

\begin{theorem}[L{\'e}vy--Khintchine formula]\label{thm_LKFormula}
Let $\X$ be a L{\'e}vy process in $G$ with triplet $(A, B, \Pi)$. Suppose that for some $1 \leq p < N+1$, $\norm{\X}_{\pvar;[0,T]} < \infty$ a.s.. 
Let $\phi : \supp(\Pi) \mapsto C_o^{\pvar}([0,1],G)$ be a $p$-approximating, endpoint continuous path function defined on $\supp(\Pi)$.

Then for every finite-dimensional complex Hilbert space $H$ and $f \in \LLL(\R^d,\uu(H))$, it holds that the function
\[
f_\phi - \id_H - \sum_{i=1}^m f(u_i)\xi_i : \supp(\Pi) \mapsto \LLL(H,H)
\]
is $\Pi$-integrable, and that
\begin{equation}\label{eq:LKFormula}
\EEE{f(S(\X^\phi)_{0,T})} = \exp\left(T\Psi_\X(f) \right),
\end{equation}
where
\begin{multline*}
\Psi_\X(f):= \sum_{i=1}^m B^if(u_i) + \frac{1}{2} \sum_{\substack{1 \leq i,j \leq m \\ d_i + d_j \leq N}} A^{i,j} f(u_i)f(u_j) \\
+ \int_{G} \left[f_\phi(x) - \id_H - \sum_{i=1}^m \xi_i(x)f(u_i) \right]\Pi(dx).
\end{multline*}
\end{theorem}

\begin{remark}
Note that every pair $(\X,\phi)$ as in Theorem~\ref{thm_LKFormula} naturally gives rise to a convolution semigroup $(\mu_t)_{t > 0}$ of probability measures on $G(\R^d)$ (which we recall is a Polish but, if $d > 1$, non-locally compact group,~\cite{ChevyrevLyons16}) given by $\mu_t = \Law{S(\X^\phi_{[0,t]})_{0,t}}$, where $\X_{[s,t]}^\phi \in C([s,t],G)$ denotes the connecting map applied to the restriction $\restr{\X}{[s,t]}$. Moreover, treating $\phi$ as a map $\supp(\Pi) \mapsto G(\R^d)$, $x \mapsto S(\phi(x))_{0,1}$, and every $f \in \LLL(\R^d,\uu(H))$ as unitary representation of $G(\R^d)$, Theorem~\ref{thm_LKFormula} bears a close resemblance to other forms of the L{\'e}vy--Khintchine formula stated in terms of unitary representations of Lie groups (see, e.g.,~\cite[Section~5.5]{Applebaum14}).
\end{remark}

\begin{remark}
Theorem~\ref{thm_LKFormula} can be seen as an extension of a related result on the expected signature of a L{\'e}vy $p$-rough path for $1 \leq p < 3$ (\cite[Theorem~53]{FrizShekhar12}) in which $\phi$ is taken as the log-linear path function $\phi(e^x) = e^{tx}$, $\forall x \in \g$, and additional moment assumptions on the L{\'e}vy measure are required to ensure existence of the expected signature.
\end{remark}

We first record the following estimate which is readily derived from standard Euler approximations to RDEs (\cite[Corollary~10.15]{FrizVictoir10}).

\begin{lemma}\label{lem_dominatedByPhi}
Let $1 \leq p < \gamma < N+1$, $\theta$ a scaling function on $G$, $W \subseteq G$ a subset, and $\phi$ a path function defined on $W$ such that $\lim_{x \rightarrow 1_G}\norm{\phi(x)}^\gamma_{\pvar;[0,1]}/\theta(x) = 0$.

Then for all $f \in \LLL(\R^d, \LLL(\R^e,\R^e))$, it holds that
\begin{equation*}
\lim_{x \rightarrow 1_G} \frac{1}{\theta(x)}\norms{f_\phi(x) - \id_{\R^e} - \sum_{i=1}^m \xi_i(x) f(u_i)
- \frac{1}{2}\sum_{\substack{1 \leq i,j \leq m \\ d_i + d_j \leq N}} \xi_i(x) \xi_j(x)
f(u_i)f(u_j)} = 0.
\end{equation*}
\end{lemma}

\begin{proof}[Proof of Theorem~\ref{thm_LKFormula}]
Without loss of generality, we cam assume $T=1$. Let $V$ be a bounded neighbourhood of $1_G$ and $W := \supp(\Pi) \cup V$. Note that $\phi$ shrinks on the diagonal (see Remark~\ref{remark_leftInvShrinks}), so we can find a path function $\psi : W \mapsto C([0,1],G)$ which is also $p$-approximating and shrinks on the diagonal and such that $\psi \equiv \phi$ on $\supp(\Pi)$ (e.g., let $\psi(x)$ be a geodesic from $1_G$ to $x$ for all $x \in V\setminus\supp(\Pi)$).

Let $X_{n1}, \ldots, X_{nn}$ be the iid array constructed in Section~\ref{subsec:approxWalk} associated to $\X$, and let $\X^n$ be the associated random walk. Due to the shrinking support of the random variables $Y_{nj}$ from Section~\ref{subsec:approxWalk}, observe that 
\begin{equation}\label{eq_allXnLarge}
\textnormal{for every $\varepsilon > 0$, $\PPP{\X^n \notin \supp(\Pi)^\varepsilon} = 0$ for all $n$ sufficiently large},
\end{equation}
where we recall the notation $\supp(\Pi)^\varepsilon$ from Section~\ref{subsec:connectFuncSkorokhod}. In particular, for all $n$ sufficiently large, $\X^n \in W^0$ a.s., so that $\X^{n,\psi}$ is well-defined.
Observe that, due to~\eqref{eq_allXnLarge} and Proposition~\ref{prop_convFromDToC}, $\X^{n, \psi} \convd \X^{\psi}$ as $C_o([0,1], G)$-valued random variables.

Let $p < p' < N+1$. Since $\norm{\X}_{\pvar;[0,1]} < \infty$ a.s. by assumption, we deduce from Theorem~\ref{thm_LevyFinitepVar}, Lemma~\ref{lem_associatedWalkTight}, and Proposition~\ref{prop:finiteVarLim}, that
\[
\X^{n,\psi} \convd \X^{\psi} \eqd \X^\phi \; \; \textnormal{as $C_o^{0,\pprimevar}([0,1],G)$-valued random variables,}
\]
where the equality in law follows from the fact that $\psi\equiv\phi$ on $\supp(\Pi)$ and $\X \in \supp(\Pi)^0$ a.s. (see Remark~\ref{remark:jumpsInSupp}).

For all $i \in \{1,\ldots, m\}$, define $q_i := 2\wedge (p/d_i)$, and let $\theta$ be a scaling function on $G$ such that $\theta \equiv \sum_{i=1}^m |\xi_i|^{q_i}$ in a neighbourhood of $1_G$. 
It follows from Lemma~\ref{lem_thetaScales} and part~\ref{point_infinitepVar} of Theorem~\ref{thm_LevyFinitepVar} that $\theta$ scales the array $X_{nj}$.

Since $\psi$ is $p$-approximating, it holds that $\lim_{x \rightarrow 1_G}\norm{\psi(x)}_{\pvar;[0,1]}^\gamma/\theta(x) = 0$ for all $\gamma > p$.
For $f \in \LLL(\R^d, \uu(H))$, observe that $f_{\psi}$ is a map from $W$ to the unitary operators on  $H$ (thus bounded) and is continuous on $\supp(\Pi)$. Since $f_{\psi} \equiv f_{\phi}$ on $\supp(\Pi)$, it now follows from Lemmas~\ref{lem_dominatedByPhi} and~\ref{lem_contOnSupport} that
\[
\lim_{n \rightarrow \infty} n\EEE{f_{\psi}(X_{n1}) - \id_H} = \Psi_\X(f),
\]
and thus
\[
\lim_{n \rightarrow \infty} \EEE{f_{\psi}(X_{n1})}^n =  \exp(\Psi_\X(f)).
\]
Since the array $X_{nj}$ is iid, note that for all $n \geq 1$
\[
\EEE{f(S(\X^{n,\psi})_{0,1})} = \EEE{f_{\psi}(X_{n1})}^n.
\]
Since $\X^{n,\psi} \convd \X^\phi$ as $WG\Omega_{p'}(\R^d)$-valued r.v.'s, and $\x \mapsto f(S(\x)_{0,1})$ is a continuous bounded function on $WG\Omega_{p'}(\R^d)$, we obtain~\eqref{eq:LKFormula}.
\end{proof}

\begin{appendix}

\section{Path Functions}\label{appendix:PathFuncs}

In this section we introduce and study the basic properties of path functions, which serve to systematically connect the left- and right-limits of a c{\`a}dl{\`a}g path.
Throughout the section, let $(E,d)$ be a metric space and equip $C([0,T],E)$ and $D([0,T], E)$ with the uniform and the Skorokhod topology respectively.

\begin{definition}\label{def_pathFunc}
For a subset $J \subseteq E \times E$, we call $\phi : J \mapsto C([0,1], E)$ a \emph{path function} defined on $J$ if
\[
\phi(x,y)_0 = x \; \textnormal{ and } \; \phi(x,y)_1 = y, \; \forall (x,y) \in J.
\]
We say $\phi$ is \emph{endpoint continuous} if $\phi$ is continuous and
\[
\phi(x,x)_t = x, \; \forall(x,x) \in J, \; \forall t \in [0,1].
\]
For $p \geq 1$, we say $\phi$ is \emph{$p$-approximating} if for every $r > 0$ there exists $C > 0$ such that for all $(x,y) \in J$ with $d(x,y) < r$
\[
\norm{\phi(x,y)}_{\pvar;[0,1]} \leq C d(x,y).
\]
When $E$ is a Lie group, we say $\phi$ is \emph{left-invariant} if there exists a subset $B \subseteq E$ such that $J = \left\{(x,y) \midset x^{-1}y \in B\right\}$ and
\[
\phi(x,y)_t = x\phi(1_E,x^{-1}y)_t, \; \forall (x,y) \in J, \; \forall t \in [0,1].
\]
\end{definition}

Note that for a Lie group $G$ and a subset $B \subseteq G$, there is a canonical bijection between functions $\phi : B \mapsto C([0,1], G)$, for which
\begin{equation}\label{eq_leftInvPathFunc}
\phi(x)_0 = 1_G \; \textnormal{ and } \; \phi(x)_1 = x, \; \forall x \in B,
\end{equation}
and left-invariant path functions defined on $J := \{(x,y) \mid x^{-1}y \in B\}$. Whenever we speak of a path function $\phi$ defined on a subset $B\subseteq G$, we shall mean that $\phi$ satisfies~\eqref{eq_leftInvPathFunc} and shall identify $\phi$ with the corresponding left-invariant path function defined on $\{(x,y) \mid x^{-1}y \in B\}$.

\subsection{The connecting map on the Skorokhod space}\label{subsec:connectFuncSkorokhod}

For a path $\x \in D([0,T], E)$ and a time $t \in [0,T]$, define $\Delta\x_t := (\x_{t-},\x_t) \in E \times E$ and $\norm{\Delta\x_t} := d(\x_{t-},\x_t)$. We call $t$ a jump time of $\x$ if $\norm{\Delta\x_t} > 0$.

For a subset $J \subseteq E \times E$ and $\varepsilon \geq 0$ define the subset of c{\`a}dl{\`a}g paths
\[
J^\varepsilon := \left\{\x \in D([0,T], E) \midset  \forall t \in [0,T], \norm{\Delta \x_t} > \varepsilon \Rightarrow \Delta \x_t \in J\right\}.
\]
In particular, $\x \in J^0$ if and only if all the jumps of $\x$ are in $J$. In the case that $E$ is a Lie group and $B \subseteq E$, we set $B^\varepsilon := J^\varepsilon$ where $J := \left\{(x,y) \midset x^{-1}y \in B\right\}$.

For a path function $\phi : J \mapsto C([0,1], E)$, we now define a map $\x \mapsto \x^\phi$ from $J^0$ to $C([0,T], E)$.
The construction is similar to the method considered in~\cite{FrizShekhar12} and~\cite{Williams01} of adding fictitious times over which to traverse the jumps.

Consider $\x \in J^0$. Let $t_1, t_2, \ldots$ be the jump times of $\x$ ordered so that $\norm{\Delta \x_{t_1}} \geq \norm{\Delta \x_{t_2}} \geq \ldots$ with $t_j < t_{j+1}$ if $\norm{\Delta \x_{t_j}} = \norm{\Delta \x_{t_{j+1}}}$. Let $0 \leq m \leq \infty$ be the number of jumps of $\x$. We call the sequence $(t_j)_{j=1}^m$ the \emph{canonically ordered jump times} of $\x$.

We henceforth fix a strictly decreasing sequence $(r_i)_{i=1}^\infty$ of positive real numbers such that $\sum_{i=1}^\infty r_i < \infty$. Define the sequence $(n_k)_{k = 0}^m$ by $n_0 = 0$, and for $1 \leq k < m+1$ let $n_k$ be the smallest integer such that $n_k > n_{k-1}$ and $r_{n_k} < \norm{\Delta\x_{t_{k}}}$.

Let $r := \sum_{k=1}^m r_{n_k}$. Define the strictly increasing (c{\`a}dl{\`a}g) function
\[
\tau : [0,T] \mapsto [0, T+r], \; \; \tau(t) = t + \sum_{k=1}^m r_{n_k} \1{t_k \leq t}.
\]
Note that $\tau(t-) < \tau(t)$ if and only if $t = t_k$ for some $1 \leq k < m+1$. Moreover, note that the interval $[\tau(t_k-), \tau(t_k))$ is of length $r_{n_k}$.

We now define $\widehat \x \in C([0,T+r],E)$ by
\[
\widehat \x_s =
\begin{cases} \x_t &\mbox{if $s = \tau(t)$ for some $t \in [0,T]$,} \\ 
\phi(\x_{t_k-}, \x_{t_k})_{(s-\tau(t_k-))/r_{n_k}} &\mbox{if $s \in [\tau(t_k-), \tau(t_k))$ for some $1 \leq k < m+1$}.
\end{cases}
\]
Denote by $\tau_r$ the linear bijection from $[0,T]$ to $[0,T+r]$. We finally define
\[
\x^\phi := \widehat \x \circ \tau_r \in C([0,T], E)
\]
and the associated time change
\[
\tau_\x := \tau_r^{-1} \circ \tau : [0,T] \mapsto [0,T].
\]
We call the map $\x \mapsto \x^\phi$ from $J^0$ to $C([0,T], E)$ the \emph{connecting map}.

\subsection{Measurability and continuity}

The main result of this subsection is a continuity property of the connecting map, which we summarise in Proposition~\ref{prop_convFromDToC}.

For a subset $J \subseteq E \times E$, we say that a path function $\phi : J \mapsto C([0,T], E)$ \emph{shrinks on the diagonal} if for every bounded set $B \subseteq J$ and $\varepsilon > 0$, there exists $\delta > 0$ such that for all $(x,y) \in B$ with $d(x,y) < \delta$
\[
\sup_{t \in [0,1]} d(\phi(x,y)_t, y) < \varepsilon.
\]

\begin{remark}\label{remark_leftInvShrinks}
Observe that every left-invariant, endpoint continuous path function defined on a subset of a Lie group shrinks on the diagonal.
\end{remark}

\begin{lemma}\label{lem_contFromSkor}
Consider $J\subseteq E\times E$ and a path function $\phi : J \mapsto C([0,1], E)$ which shrinks on the diagonal. Suppose that $\phi$ is endpoint continuous on a subset $K \subseteq J$.

Let $\x \in K^0$ and a sequence $\x(n) \in J^0$ such that $\x(n) \rightarrow \x$ in the Skorokhod topology as $n \rightarrow \infty$, and such that for every $\varepsilon > 0$, there exists $n_0 > 0$ such that $\x(n) \in K^\varepsilon$ for all $n \geq n_0$.

Then $\lim_{n \rightarrow \infty} d_{\infty;[0,T]}(\x(n)^\phi, \x^\phi) = 0$.
\end{lemma}

\begin{proof}
Let $\varepsilon > 0$. By uniform continuity of $\x^\phi$, there exists $\eta > 0$ such that
\[
\sup_{|s-t| < \eta} d(\x^\phi_t, \x^\phi_s) < \varepsilon.
\]

Let $t_1, t_2, \ldots$ be the canonically ordered jump times of $\x$, and denote $[s_i, u_i) := [\tau_\x(t_i-), \tau_\x(t_i))$. For another element $\y\in D([0,T], E)$, let $t_1', t_2', \ldots$ be the jump times of $\y$ and $[s_i', u_i') := [\tau_\y(t_i'-), \tau_\y(t_i'))$. Let $(n_k)_{k=0}^m$ and $(n_k')_{k=1}^{m'}$ be the corresponding sequences for $\x$ and $\y$ respectively (where $m$ and $m'$ are the number of jumps of $\x$ and $\y$).

Denote by $\Lambda^*$ the set of continuous, strictly increasing bijections $\lambda : [0,T] \mapsto [0,T]$ and let $\id \in \Lambda^*$ be the identity map. Consider on $D([0,T],E)$ the Skorokhod metric
\[
\sigma(\x,\y) := \inf_{\lambda \in \Lambda^*} \max\{d_{\infty;[0,T]}(\lambda, \id), d_{\infty;[0,T]}(\x\circ \lambda, \y)\}.
\]

Let $\delta > 0$ (which we shall send to zero), and suppose that there exists $\lambda \in \Lambda^*$ such that $d_{\infty;[0,T]}(\x \circ \lambda, \y) < \delta$ and $d_{\infty;[0,T]}(\lambda,\id) < \delta$.

Observe that there exists an integer $k \geq 1$ such that $\lambda(t_i') = t_i$ for all $i \in \{1, \ldots, k\}$, and, denoting by $v_1 < \ldots < v_k$ (resp. $v_1' < \ldots < v_k'$) the same set of points as $t_1, \ldots, t_k$ (resp. $t_1', \ldots, t_k'$) ordered monotonically, it holds that $\lambda(t') \in [v_i, v_{i+1})$ for all $t' \in [v_i', v_{i+1}')$.  In particular, it holds that $d(\y_{t_i'-},\x_{t_i-}), d(\y_{t_i'}, \x_{t_i}) < \delta$ for all $i \in \{1, \ldots, k\}$.

Moreover, by choosing $\delta$ sufficiently small, we can assume that $n_i = n_i'$ for all $i \in \{1,\ldots, k\}$ and that $k$ is sufficiently large so that, by making $\sum_{j = k+1}^\infty r_j$ sufficiently small, it holds that $|\tau_{\y}(t') - \tau_\x(\lambda(t'))| < \eta$ for all $t' \in [v_i', v_{i+1}')$ (this is where we have used the condition $r_{n_j} < \norm{\Delta\x_{t_j}}$).

In particular, it holds that that for all $t' \in [v_i', v_{i+1}')$
\begin{equation}\label{eq_dyx}
\begin{split}
d(\y^\phi_{\tau_{\y}(t')}, \x^\phi_{\tau_{\y}(t')})
&\leq d(\y^\phi_{\tau_{\y}(t')}, \x^\phi_{\tau_\x(\lambda(t'))}) + d(\x^\phi_{\tau_\x(\lambda(t'))}, \x^\phi_{\tau_{\y}(t')}) \\
&< d(\y_{t'}, \x_{\lambda(t')}) + \varepsilon \\
&< \delta + \varepsilon.
\end{split}
\end{equation}
This covers all points not in an interval $[s_i',u_i')$ (note that no assumptions on $\phi$ were needed yet except that $\phi(x,y)$ itself is a continuous path for each $(x,y) \in J$).

Now we let $\y = \x(n)$ for some $n$. We may choose $n$ sufficiently large, such that $\sigma(\x,\y) < \delta$ and such that $\Delta\y_{t_i'} \in K$ for all $i \in \{1, \ldots, k\}$.

Due to the continuity of $\phi : K \mapsto C([0,1], E)$ at $\Delta\x_{t_i} \in K$, it follows that for all $w' \in [s_i', u_i')$ and $i \in \{1,\ldots, k\}$, there exists $w \in [s_i, u_i)$ such that $|w' - w| < \eta$ and $d(\x^\phi_w, \y^\phi_{w'}) < \varepsilon$ , so that
\[
d(\y^\phi_{w'}, \x^\phi_{w'}) < d(\y^\phi_{w'}, \x^\phi_w) + d(\x^\phi_w, \x^\phi_{w'}) < 2\varepsilon.
\]

Finally, since $\phi$ shrinks on the diagonal, we may further decrease $\delta$ if necessary so that for all $w' \in [s_j', u_j')$ and $j > k$, it holds that $|w' - u_j'| < \eta$ and $d(\y^\phi_{w'}, \y^\phi_{u_j'}) < \varepsilon$. 
Now $u_j' = \tau_\y(t')$ for some $t' \in [0,T]$, and thus, by~\eqref{eq_dyx},
\[
d(\y^\phi_{u_j'}, \x^\phi_{u_j'}) < \delta + \varepsilon,
\]
from which it follows that
\[
d(\y^\phi_{w'},\x^\phi_{w'}) \leq d(\y^\phi_{w'}, \y^\phi_{u_j'}) + d(\y^\phi_{u_j'}, \x^\phi_{u_j'}) + d(\x^\phi_{u_j'}, \x^\phi_{w'}) < \delta + 3\varepsilon.
\]
\end{proof}

In the following proposition, we equip all topological spaces with their respective Borel $\sigma$-algebras.

\begin{proposition}\label{prop_convFromDToC}
Suppose $E$ is a Polish space, $K \subseteq J \subseteq E\times E$ are measurable sets, and $\phi : J \mapsto C([0,1],E)$ is a measurable path function.

\begin{enumerate}[label={(\roman*)}]
\item \label{point:Meas} Then $J^\varepsilon$ is a measurable subset of $D([0,T], E)$ for all $\varepsilon \geq 0$, and the connecting map $\cdot^\phi : J^0 \mapsto C([0,T], E)$ is measurable.

\item \label{point:Cont} Let $\X$ be a $D([0,T], E)$-valued random variable such that $\X \in K^0$ a.s.. Let $(\X_n)_{n \geq 1}$ be a collection of $D([0,T],E)$-valued random variables such that $\X_n \in J^0$ a.s. and $\X_n\convd \X$ as $D([0,T],E)$-valued random variables.
Suppose further that for every $\varepsilon > 0$, $\lim_{n \rightarrow \infty}\PPP{\X_n \notin K^\varepsilon} = 0$, and that $\phi$ is endpoint continuous on $K$ and shrinks on the diagonal in $J$.

Then $\X_n^\phi \convd \X^\phi$ as $C([0,T],E)$-valued random variables.
\end{enumerate} 
\end{proposition}

\begin{proof}
\ref{point:Meas} Measurability of $J^\varepsilon$ is an easy consequence of the measurability of the finite-dimensional projections $\x \mapsto \x_t$ on $D([0,T],E)$ (\cite[Theorem~12.5]{Billingsley99}). To show measurability of the connecting map, note that it suffices to show that for every $s \in [0,T]$, the map $\x \mapsto \x^\phi(s)$ is measurable (cf.~\cite[Example~1.3]{Billingsley99}), which in turn follows easily from the construction of $\x^\phi$.

\ref{point:Cont}
Note that the condition $\lim_{n \rightarrow \infty}\PPP{\X_n \notin K^\varepsilon} = 0$ for all $\varepsilon > 0$ implies that
\[
\Y_n := (\X_n, \1{\X_n \notin K^{1}}, \1{\X_n \notin K^{1/2}}, \1{\X_n \notin K^{1/4}}, \ldots)
\]
is a sequence of $D([0,T],E) \times \{0,1\}^{\N}$-valued random variables which converges in law to $\Y := (\X, 0, 0, \ldots)$.
The conclusion now follows from an application of the Skorokhod representation theorem~\cite[Theorem~3.30]{Kallenberg97} and Lemma~\ref{lem_contFromSkor}.
\end{proof}

\subsection{$p$-variation}\label{subsec:pVarConnect}

The main result of this subsection is Proposition~\ref{prop_pvarPsi}, which shows that a $p$-approximating path function does not significantly increase the $p$-variation of a c{\`a}dl{\`a}g path.

We first require the following lemma, whose proof was inspired by~\cite[Lemma~22]{FrizShekhar12}.

\begin{lemma}\label{lem:xpvarBound}
Let $(E,d)$ be a metric space and $\x : [0,T] \mapsto E$ a function (not necessarily c{\`a}dl{\`a}g). Let $(I_n)_{n \geq 1}$ be a countable collection of disjoint open subintervals of $[0,T]$ and set $I := \cup_{n \geq 1}I_n$. Define $\y : [0,T] \mapsto E$ by $\y_t = \x_t$ for $t \in [0,T]\setminus I$, and $\y_t = \x_{c_n}$ for $t \in (c_n,d_n) := I_n$. 
Let $p > 0$ and $C = 1+2\cdot 2^{(p-1)\vee 0} + 3^{(p-1)\vee 0}$. Then
\[
\norm{\x}^p_{\pvar;[0,T]} \leq (C+3^{(p-1)\vee 0})\sum_{n=1}^\infty \norm{\x}^p_{\pvar;[c_n,d_n]} + C\norm{\y}^p_{\pvar;[0,T]}.
\]
\end{lemma}

\begin{proof}
Define the super-additive functions $\omega_\x(s,t) = \norm{\x}_{\pvar;[s,t]}^p$ and $\omega_\y(s,t) = \norm{\y}_{\pvar;[s,t]}^p$.
Let $\DD = (t_0, t_1, \ldots, t_k)$ be a partition of $[0,T]$.

Denote by $J_1 = (a_1,b_1), \ldots, J_m = (a_m, b_m)$ those intervals $I_n$ which contain some partition point $t_j \in \DD$, ordered so that $b_j < a_{j+1}$ for all $j \in \{1,\ldots m-1\}$. Call a \emph{block} a consecutive run of partition points $t_j, t_{j+1}, \ldots, t_n$ either all in $J_r$ for some $r \in \{1,\ldots, m\}$, in which case we call it red, or all outside $I$, in which case we call it blue. We call a consecutive pair of partition points $t_i, t_{i+1} \in \DD$ which lie in different blocks either red-red, red-blue, or blue-red depending on their respective blocks (note there are no blue-blue pairs). For convenience of notation, set $J_0 = (a_0,b_0) := (-\infty, 0)$ and $J_{m+1} = (a_{m+1},b_{m+1}) := (T, \infty)$.

For $r \in \{1,\ldots, m\}$ and a red block $t_j, t_{j+1}, \ldots, t_n$ in $J_r$ we have
\[
\sum_{i=j+1}^{n} d(\x_{t_{i-1}}, \x_{t_i})^p \leq \omega_\x(a_r,b_r).
\]
For $r \in \{0,\ldots, m\}$ and a blue block $t_j, t_{j+1}, \ldots, t_n$ between $J_r, J_{r+1}$ we have
\[
\sum_{i=j+1}^{n} d(\x_{t_{i-1}}, \x_{t_i})^p
= \sum_{i=j+1}^{n} d(\y_{t_{i-1}}, \y_{t_i})^p
\leq \omega_\y(b_r,a_{r+1}).
\]
For $r \in \{1,\ldots, m\}$ and a red-blue pair $t_i, t_{i+1}$ with $t_i \in J_r$ and $t_{i+1}$ between $J_r,J_{r+1}$, we have
\begin{align*}
d(\x_{t_i}, \x_{t_{i+1}})^p
&\leq 2^{(p-1)\vee 0}\left[d(\x_{t_i}, \x_{b_r})^p + d(\x_{b_r}, \x_{t_{i+1}})^p\right] \\
&\leq 2^{(p-1)\vee 0}\left[\omega_\x(a_r,b_r) + \omega_\y(b_r, t_{i+1})\right].
\end{align*}
For $r \in \{0,\ldots, m-1\}$ and a blue-red pair $t_i, t_{i+1}$ with $t_i$ between $J_{r},J_{r+1}$ and $t_{i+1} \in J_{r+1}$, we have
\begin{align*}
d(\x_{t_i}, \x_{t_{i+1}})^p
&\leq 2^{(p-1)\vee 0}\left[d(\x_{t_i}, \x_{a_{r+1}})^p + d(\x_{a_{r+1}}, \x_{t_{i+1}})^p\right] \\
&\leq 2^{(p-1)\vee 0}\left[\omega_\y(t_i,a_{r+1}) + \omega_\x(a_{r+1}, b_{r+1})\right].
\end{align*}
Finally for $r \in \{1,\ldots, m-1\}$ a red-red pair $t_i, t_{i+1}$ with $t_i \in J_r$ and $t_{i+1} \in J_{r+1}$, we have
\begin{align*}
d(\x_{t_i}, \x_{t_{i+1}})^p
&\leq 3^{(p-1)\vee 0}\left[d(\x_{t_i}, \x_{b_r})^p + d(\x_{b_r}, \x_{a_{r+1}})^p + d(\x_{a_{r+1}}, \x_{t_{i+1}})^p\right] \\
&\leq 3^{(p-1)\vee 0}\left[\omega_\x(a_r, b_r) + \omega_\y(b_r,a_{r+1}) + \omega_\x(a_{r+1}, b_{r+1})\right].
\end{align*}

Since $\omega_\x$ and $\omega_\y$ are super-additive, the conclusion now follows from splitting the sum $\sum_{i=1}^k d(\x_{t_{i-1}}, \x_{t_i})^p$ into blocks and consecutive pairs in different blocks.
\end{proof}

\begin{corollary}\label{cor_pvarNotIncrease}
Let $p \geq 1$, $J \subseteq E\times E$, $\phi : J \mapsto C([0,T], E)$ a path function, and $\z \in J^0$. Suppose that there exists $C > 0$ such that $\norm{\phi(\z_{t-},\z_t)}_{\pvar;[0,1]} \leq Cd(\z_{t-},\z_t)$ for every jump time $t$ of $\z$.
Then
\[
\norm{\z^\phi}^p_{\pvar;[0,T]} \leq \left( 1+2^p+3^{p-1} + C^p(1+2^p+2\cdot 3^{p-1}) \right)\norm{\z}^p_{\pvar;[0,T]}.
\]
\end{corollary}

The following result now follows immediately from Corollary~\ref{cor_pvarNotIncrease} and from the definition of a $p$-approximating path function (Definition~\ref{def_pathFunc}).

\begin{proposition}\label{prop_pvarPsi}
Let $p \geq 1$, $J \subseteq E\times E$, and $\phi : J \mapsto C([0,T], E)$ a $p$-approximating path function. There exists a continuous function $\psi : [0,\infty) \mapsto [0,\infty)$ such that for some $R,\varepsilon > 0$, $\psi(x) = Rx$ for all $x \in [0,\varepsilon)$, and such that $\norm{\x^\phi}_{\pvar;[0,T]} \leq \psi(\norm{\x}_{\pvar;[0,T]})$ for all $\x \in J^0$.
\end{proposition}

\section{Infinite $p$-variation of L{\'e}vy processes in Lie groups}\label{appendix:InfpVar}

The purpose of this section is to establish conditions under which sample paths of a L{\'e}vy process have \emph{infinite} $p$-variation. The methods are all well-known for the case $G = \R^d$, so we mostly provide indications of how they extend to a general Lie group.
Throughout this section, we use the notation from Section~\ref{sec:LieGroupIidArrays}.

Let $\X$ be a L{\'e}vy process in a Lie group $G$ with triplet $(A, B, \Pi)$ and let $\tau = \inf\left\{t \geq 0 \midset \X_t \notin U\right\}$ be the first exit time of $\X$ from $U$.
Let $y_0 \in \g\setminus \log(U)$ be a distinguished point and consider the $\g$-valued process
\[
\Y_t = 
\begin{cases} \log(\X_t) &\mbox{if } t \in [0,\tau) \\ 
\Y_t = y_0 &\mbox{if } t \geq \tau.
\end{cases}
\]
For $\varepsilon > 0$, define the $\g$-valued process $\Y^\varepsilon_t := \varepsilon^{-1/2}(\Y_{\varepsilon t})$ for $t \in [0,1]$. Let $\B$ be a $\g$-valued centred Brownian motion starting from zero with covariance matrix $(A^{i,j})_{i,j=1}^m$ with respect to the basis $u_1,\ldots, u_m$.

\begin{lemma}\label{lem:convToB}
It holds that $\Y^\varepsilon \xrightarrow[\varepsilon \rightarrow 0]{\DD} \B$ as $D_o([0,1], \g)$-valued random variables.
\end{lemma}

\begin{proof}
Note that for every $\varepsilon > 0$, $\Y^\varepsilon$ can be considered as a $\g$-valued Markov process (for which every point outside $\varepsilon^{-1/2}\log(U)$ is absorbing). Writing $L^\varepsilon$ and $L_\B$ for the generators of $\Y^\varepsilon$ and $\B$ respectively, it suffices to show that $L^\varepsilon f \rightarrow L_\B f$ in $C_0(\g)$ for all $f \in C^\infty_c(\g)$ (see, e.g.,~\cite[Chapter~17]{Kallenberg97}). This in turn follows from writing the generator of $\X$ in the $\log$ chart and performing a straightforward limiting argument.
\end{proof}

\begin{proposition}\label{prop_xiI}
Suppose $A^{i,i} > 0$ for some $i \in \{1,\ldots, m\}$. Then
\[
\PPP{\limsup_{t \rightarrow 0} t^{-1/2}|\xi_i(\X_t)| = \infty} = 1.
\]
\end{proposition}

\begin{proof}
Let $c > 0$. Lemma~\ref{lem:convToB} implies that there exist $\delta, \varepsilon_0 > 0$ such that for all $0 < \varepsilon \leq \varepsilon_0$
\[
\PPP{\varepsilon^{-1/2}|\xi_i(\X_\varepsilon)| > c} > \delta.
\]
Observe that, by the CBH formula, there exist a neighbourhood $V$ of $1_G$ and a constant $C > 0$ such that for all $x,y \in V$
\begin{equation}\label{eq_boundinV}
|\xi_i(x^{-1}y)| \leq |\xi_i(y) - \xi_i(x)| + C|\xi(x)|^2 + C|\xi(y)|^2.
\end{equation}
For $k \geq 1$, define $\varepsilon_k := 3^{-k}\varepsilon_0$. Then by Lemma~\ref{lem:convToB} and the first Borel–Cantelli lemma there exists a strictly increasing sequence $(k(n))_{n \geq 1}$ such that
\begin{equation}\label{eq_PlimsupC}
\PPP{\limsup_{n \rightarrow \infty} \varepsilon_{k(n)}^{-1/2}C\left(|\xi(\X_{\varepsilon_{k(n)}})|^2 + |\xi(\X_{\varepsilon_{k(n)-1}})|^2\right) > c/2} = 0.
\end{equation}
On the other hand, since $\X_{\varepsilon_k,\varepsilon_{k-1}} \eqd \X_{\varepsilon_{k-1} - \varepsilon_k}$ and the r.v.'s $(\X_{\varepsilon_k, \varepsilon_{k-1}})_{k \geq 1}$ are independent, the second Borel–Cantelli lemma yields
\[
\PPP{\limsup_{n \rightarrow \infty} \varepsilon_{k(n)}^{-1/2}|\xi_i(\X_{\varepsilon_{k(n)}, \varepsilon_{k(n)-1}})| > c } = 1.
\]
Since $\lim_{t \rightarrow 0}\X_t = 1_G$ a.s., we now readily deduce from~\eqref{eq_boundinV},~\eqref{eq_PlimsupC}, and the definition of $\varepsilon_k$ that
\[
\PPP{\limsup_{k \rightarrow \infty} \varepsilon_{k}^{-1/2}|\xi_i(\X_{\varepsilon_{k}})| > \frac{c}{4\sqrt 3}} = 1.
\]
As $c>0$ was arbitrary, the conclusion follows.
\end{proof}

\begin{corollary}\label{cor_xiInfVar}
Suppose $A^{i,i} > 0$ for some $i \in \{1,\ldots, m\}$. Then
\[
\PPP{\sup_{\DD \subset [0,1]} \sum_{t_k \in \DD} |\xi_i(\X_{t_k,t_{k+1}})|^2 = \infty} = 1.
\]
\end{corollary}

\begin{proof}
By Proposition~\ref{prop_xiI}, $\limsup_{t \rightarrow 0} t^{-1}|\xi_i(\X_t)|^2 = \infty$ a.s..
Since $\X$ has stationary and independent increments, the conclusion follows from an application of the Vitali covering argument (see~\cite[Proposition p. 68]{Bretagnolle72II}, or~\cite[Theorem~13.69]{FrizVictoir10}).
\end{proof}

The following is a form of the classical Blumenthal--Getoor index~\cite{BlumenthalGetoor61} adapted to the setting of Lie groups. Recall the definitions of $\Gamma_i$ and $K$ from Section~\ref{subsec:approxWalk}.

\begin{proposition}[Blumenthal--Getoor index]\label{prop_BlumenthalGetoor} Let $i \in \{1,\ldots, m\}$ and $q > 0$. Then
\begin{equation}\label{eq_infqXiVar}
\PPP{\sup_{\DD \subset [0,1]} \sum_{t_k \in \DD} |\xi_i(\X_{t_k,t_{k+1}})|^q = \infty} = 1
\end{equation}
if either
\begin{enumerate}[label={(\roman*)}]
\item \label{point_qInGamma} $q \in \Gamma_i$, or

\item \label{point_iInK} $i \in K$ and $q < 1$.
\end{enumerate}
\end{proposition}

\begin{proof}
Define $f \in C_c(G)$ by $f(x) = 1-\exp(-|\xi_i(x)|^q)$. Since $\X$ has independent and stationary increments, we can readily show (cf.~\cite[p. 499]{BlumenthalGetoor61}) that~\eqref{eq_infqXiVar} holds whenever
\begin{equation}\label{eq_fXInf}
\lim_{t \rightarrow 0}t^{-1}\EEE{f(\X_t)} = \infty.
\end{equation}
It thus suffices to show that~\eqref{eq_fXInf} holds in both cases of~\ref{point_qInGamma} and~\ref{point_iInK}:

\ref{point_qInGamma} Let $(\psi_n)_{n \geq 1}$ be a non-decreasing sequence of non-negative functions in $C^\infty(\R)$, each vanishing on some neighbourhood of zero, and such that $\lim_{n \rightarrow \infty}\psi_n (x) = |x|^{q}$ for all $x \in \R$. Then for $f_n(x) := 1-\exp(-\psi_n(\xi_i(x)))$, we have
\[
\lim_{t \rightarrow 0} t^{-1}\EEE{f_n(\X_t)}
= \int_G f_n(x) \Pi(dx)
\geq c\int_G \psi_n(\xi_i(x)) \Pi(dx) \xrightarrow[n \rightarrow \infty]{} \infty,
\]
where the final convergence follows from $q \in \Gamma_i$. Since $0 \leq f_n \leq f$, we obtain~\eqref{eq_fXInf}.

\ref{point_iInK} Since $q < 1$, for every integer $n \geq 1$ we can find $\psi_n \in C^\infty(\R)$ such that $|\psi_n(x)| \leq |x|^q$ for all $x \in \R$ and such that $\psi_n(x) = nx/\widetilde B^i$ for all $x$ in a neighbourhood $V_n$ of zero. Note that we may suppose $A^{i,i} = 0$ and $q \notin \Gamma_i$ (as otherwise the desired result follows by Corollary~\ref{cor_xiInfVar} or by case~\ref{point_qInGamma}). Then for $f_n(x) := 1-\exp(-\psi_n(\xi_i(x)))$, a straightforward calculation shows that
\[
\lim_{t \rightarrow 0}t^{-1}\EEE{f_n(\X_t)} = n \left(1 + n^{-1}\int_G f_n(x)\Pi(dx)\right) \xrightarrow[n \rightarrow \infty]{} \infty,
\]
where the final convergence follows from $q \notin \Gamma_i$ and $|f_n(x)| \leq C|\psi_n(\xi_i(x))|$. Since again $f_n \leq f$, we obtain~\eqref{eq_fXInf}.
\end{proof}

\end{appendix}

\printbibliography

\end{document}